\documentclass[12pt,epsfig,amsfonts]{amsart} 
\usepackage{amsmath,amsthm,amssymb,amscd,epsfig}
\usepackage{graphicx}
\usepackage{mathrsfs}

\newtheorem*{theorema}{Theorem A}
\newtheorem*{theoremb}{Theorem B}
\newtheorem*{theoremc}{Theorem C}

\newtheorem*{thmA1}{Theorem A.1}
\newtheorem*{thmB1}{Theorem B.1}
\newtheorem*{lemB2}{Lemma B.2}
\newtheorem*{lemB3}{Lemma B.3}

\newtheorem{prop}{Proposition}[section]
\newtheorem{lemma}[prop]{Lemma}
\newtheorem{thm}[prop]{Theorem}
\newtheorem{cor}[prop]{Corollary}
\newtheorem*{cor1}{Corollary}

\theoremstyle{definition}

\theoremstyle{remark}
\newtheorem{remark}[prop]{Remark}

\numberwithin{equation}{section}

\begin{document}

\thanks{
The first-named author was supported by 
 the Grant-in-Aid for Scientific Research (C) of the JSPS, 16K05179.
 The second-named author was supported by
the Grant-in-Aid for Young Scientists (A) of the JSPS, 15H05435,
 the Grant-in-Aid for Scientific Research (B) of the JSPS, 16KT0021.
We thank Juan Rivera-Letelier and Mike Todd for fruitful discussions.
}

\title[LDP for $S$-unimodal maps with flat critical point]{Large deviation principle
 for $S$-unimodal maps with flat critical point}


\author{Yong Moo Chung}
\address{Department of Applied Mathematics,
Hiroshima University, Higashi-Hiroshima,
739-8527, JAPAN}
\curraddr{}
\email{}
\thanks{}

\author{Hiroki Takahasi}
\address{Department of Mathematics,
Keio University, Yokohama,
223-8522, JAPAN}
\curraddr{}
\email{hiroki@math.keio.ac.jp}
\thanks{}


\date{today}


\begin{abstract}
We study a topologically exact, negative Schwarzian unimodal map  
whose critical point is non-recurrent and flat.
Assuming the critical order is 
either logarithmic or polynomial,
we establish the Large Deviation Principle
and give a partial description of the zeros of the corresponding rate functions.
We apply our main results to a certain parametrized family of unimodal maps in the same topological conjugacy class, and give a complete description
of the zeros of the rate functions. We observe a qualitative change at a transition parameter,
and show that the sets of zeros depend continuously on the parameter
even at the transition.

\end{abstract}


\maketitle



\section{introduction}

Consider a dynamical system $f\colon X\to X$ of a compact topological space $X$.
The theory of large deviations deals with the behavior of the empirical mean
$$\delta_x^n=\frac{1}{n}\left(\delta_x+\delta_{f(x)}+\cdots+\delta_{f^{n-1}(x)}\right)\quad \text{as $n\to\infty$},$$
 where $\delta_{x}$ denotes the Dirac measure at $x$.
We put a Lebesgue measure $|\cdot |$ on $X$ as a reference measure, and ask the asymptotic behavior of the empirical mean 
for Lebesgue almost every initial condition.
For general accounts on the theory of large deviations, 
see for example Ellis \cite{Ell85},
 Dembo and Zeitouni \cite{DemZei98},
  Rassoul-Agha and Sepp\"al\"ainen \cite{RasSep}.

Let $\mathcal M$ denote the space of Borel probability measures on $X$
endowed with the topology of weak convergence.
We say {\it the Large Deviation Principle} (the LDP) holds if
there exists a lower semi-continuous function 
 $\mathscr{I}=\mathscr{I}(f;\, \cdot)\colon\mathcal M\to[0,\infty]$ 
 which satisfies the following:

\begin{itemize}
\item [-](lower bound) for every open subset $\mathcal G$ of $\mathcal M$,
\begin{equation}\label{low}
\liminf_{n\to\infty}\frac{1}{n} \log \left|\left\{x\in X\colon\delta_x^n\in\mathcal G\right\}\right|\geq-\inf_{\mu\in\mathcal G}\mathscr{I}(\mu);\end{equation} 

\item[-] (upper bound) for every closed subset $\mathcal K$ of $\mathcal M$,
\begin{equation}\label{up}
\limsup_{n\to\infty}\frac{1}{n} \log \left|\left\{x\in X\colon\delta_x^n\in\mathcal K\right\}\right|\leq
-\inf_{\mu\in\mathcal K}\mathscr{I}(\mu),\end{equation}
\end{itemize} 
where
$ \log 0=-\infty$,
$\inf\emptyset=\infty$
and $\sup\emptyset=-\infty.$
The function $\mathscr{I}$ is called {\it a rate function}.
If the LDP holds, then the rate function is unique,
and given by the Legendre transform of the cumulant generating function \cite{DemZei98}.

In rough terms, the LDP implies that under iteration each empirical mean gets close to the set of measures
where the rate function vanishes.
These measures are physically relevant ones, or else considered to impede transport, slow down the rate of mixing of the system.
It is important to determine the set $\{\mu\in\mathcal M\colon \mathscr{I}(\mu)=0\}$, as it corresponds to the set of those limit distributions
that represent a sub-exponentially large set of initial conditions.
Also important is to describe the (in)stability of the structure of this set under small perturbations of the system.

For a transitive uniformly hyperbolic system with H\"older continuous derivative, the LDP was established by 
Takahashi \cite{Tak87}, Orey and Pelikan \cite{OrePel89}, Kifer 
\cite{Kif90}, Young \cite{You90}.
The rate function $\mathscr{I}$ is given by 
$$\mathscr{I}(\mu)=\begin{cases}
h(\mu)-   \int\sum_{\chi_i(x)>0}\chi_i(x)d\mu(x) &\text{ if $\mu$ is $f$-invariant};\\
-\infty&\text{ otherwise,}
\end{cases}$$
where
$h(\mu)=h(f;\mu)$ denotes the Kolmogorov-Sina{\u\i} entropy of $\mu$ and $\sum_{\chi_i(x)>0}\chi_i(x)$
the sum of positive Lyapunov exponents at $x$ counted with multiplicity.
The rate function vanishes only at the Sina{\u\i}-Ruelle-Bowen measure \cite{Bow75,Rue76,Sin72},
and this measure depends continuously on the system.
The LDP gives exponential bounds on the probabilities that the empirical means stay away from the
Sina{\u\i}-Ruelle-Bowen measure.

For non-hyperbolic systems, few results on the LDP were available
until recently.
For interval maps with neutral fixed point such as the Manneville-Pomeau map \cite{PomMan80},
 Pollicott and Sharp \cite{PolSha09} proved several results closely related to the LDP
 assuming the existence of an invariant probability measure that is absolutely continuous with
 respect to the Lebesgue measure.
The method in \cite{Chu11} implies the LDP for some non-hyperbolic systems
which are very close to uniformly hyperbolic ones, such as almost Anosov systems, 
interval maps with neutral fixed point, and topologically exact
unimodal maps with non-recurrent non-flat critical point.
In \cite{ChuTak14} the LDP was established for certain non-uniformly expanding quadratic maps 
under strong assumptions on the hyperbolicity and recurrence of the orbit of the critical point.

A substantial progress has been made in \cite{ChuRivTak} in which the LDP was established for {\it every} multimodal map with non-flat critical point
and H\"older continuous derivatives that is topologically exact.
This includes maps with a very weak form of hyperbolicity, and even those with pathological behaviors found by Hofbauer and Keller \cite{HofKel90,HofKel95},
for which there is no asymptotic measure and no good statistical limit theorem was previously known. 
This universality of the LDP amidst the diversity of one-dimensional dynamics challenges
the general paradigm that good statistical limit theorems are manifestations of 
(an weak form of) hyperbolicity.
 
 The aim of this paper is to establish the LDP for unimodal maps with non-recurrent flat critical point.
In \cite{ChuRivTak} all critical points are assumed to be non-flat (See e.g., \cite{dMevSt93} for the definition),
 and this assumption is crucial as developed below.
We remove this assumption at the cost of imposing the non-recurrence of the critical point.
We also study the structure of the set of zeros of the associated rate function.

In what follows,
let $X=[0,1]$ and
 $f\colon X\to X$ be
{\it a unimodal map}, i.e., a $C^1$ map whose critical set $\{x\in X\colon Df(x)=0\}$ consists of a single point $c\in(0,1)$ that is an extremum. 
We say $f$ is {\it topologically exact} if for any open subset $U$ of $X$ there exists an integer $n\geq1$ such that $f^n(U)=X$.
Let $\mathcal M(f)$ denote the set of elements of $\mathcal M$ which are $f$-invariant.
 {\it An $S$-unimodal map} is a unimodal map of class $C^3$
on $X\setminus\{c\}$
with negative Schwarzian derivative. 
 Let $\omega(c)$ denote the omega-limit set of $c$.
 The critical point $c$ is {\it non-recurrent} if $c\notin\omega(c)$.

For an $S$-unimodal map with non-recurrent 
flat critical point (i.e., a critical point at which all derivatives vanish) having only hyperbolic repelling periodic points, Benedicks and Misiurewicz \cite{BenMis89} constructed
a $\sigma$-finite invariant measure that is absolutely continuous with respect to the Lebesgue measure.
Zweim\"uller \cite{Zwe04} proved statistical properties of the 
invariant measure,  
including a polynomial bound on 
decay of correlations for maps with a flat critical behavior 
like $\exp(-|x-c|^{-\alpha})$ $(\alpha>0)$. 
For a parametrized family of $S$-unimodal maps with
this type of critical behavior, 
Thunberg \cite{Thu99}
proved a version of Benedicks-Carleson's theorem \cite{BenCar91}: the existence of a positive measure set of parameters for which the corresponding maps
exhibit an exponential growth of derivatives along the orbit of the critical point.
This positive measure set contains a dense subset corresponding to maps with non-recurrent critical point. 
The same type of flat critical behavior, to be referred to as of polynomial order in our terms,
 was also considered by
Dobbs \cite{Dob14}.

In what follows, for a flat critical point $c$ we assume there exists
a $C^1$ function $\ell$ on $X\setminus\{c\}$ such that the following holds:

\begin{itemize}

\item[(i)] 
$ \ell(x)\to\infty$ and $|D\ell(x)|\to\infty$.
Here, $x\to c$ indicates both 
as $x\to c+0$ and $x\to c-0$;


\item[(ii)] there exist
$C^1$ diffeomorphisms $\xi$, $\eta$ of $\mathbb R$ such that
$\xi(c)=0=\eta(f(c))$ and $|\xi(x)|^{\ell(x)}=\eta( f(x))$ for all $x$ near $c$.
\end{itemize}

The function $\ell$ determines how fast $Df(x)$ goes to $0$ as $x\to c$.
For a technical reason as explained below,
we work with two specific rates of growth of $\ell$.
The flat critical point $c$ is 
 {\it of polynomial order} if there exists a $C^1$ function $v$ on $X$ such that $v(c)>0$
 and for all $x$ near $c$,
$$\ell (x) = |x-c|^{-v(x)}.$$
It is 
{\it of logarithmic order} 
if there exist a $C^1$ function $u$ on $X$ and $\alpha>0$ such that $u(c)>0$ and for all $x$ near $c$,
$$\ell(x)= u(x)| \log|x-c||^\alpha.$$
All our main results hold for topologically exact $S$-unimodal maps with non-recurrent flat critical point
that is of polynomial or logarithmic order. To simplify expositions we restrict ourselves to the case of polynomial order.

According to \cite{Chu11,ChuRivTak,ChuTak14}
we now introduce a function $\mathscr{I}\colon \mathcal M\to[0,\infty]$
which is a natural candidate for the rate function for the LDP for interval maps with critical points.
For an $S$-unimodal map $f$ and
 $\nu\in\mathcal M(f)$ define {\it a Lyapunov exponent} $\chi(\nu)=\chi(f;\nu)$ 
by $$\chi(\nu)=\int\log|Df|d\nu.$$
From the result of Bruin and Keller \cite{BruKel98},
 $\chi(\nu)\geq0$ holds\footnote{The proof does not use the non-flatness of the critical point.}
for every $\nu\in\mathcal M(f)$ provided all periodic points of $f$ are hyperbolic repelling.
Define $\mathscr{F}=\mathscr{F}(f;\, \cdot) \colon \mathcal M \to [-\infty, 0]$ by
$$ \mathscr{F}(\nu)
=
\begin{cases}h(\nu)-\chi(\nu) &\text{ if $\nu\in\mathcal M(f)$};\\
-\infty&\text{ otherwise.}\end{cases}$$
Although the entropy is upper semi-continuous, the Lyapunov exponent is not lower semi-continuous
due to the presence of the critical point, and hence
$\mathscr{F}$ is not upper semi-continuous.
To rectify this point we consider an upper semi-continuous regularization of $\mathscr{F}$.
Define 
\begin{equation}\label{rate}
 \mathscr{I}(\mu)
=
-\inf_{\mathcal G \ni \mu}\sup_{\nu\in\mathcal G}\mathscr{F}(\nu),
\end{equation}
where the infimum is taken over all open subsets~$\mathcal G$ of~$\mathcal M$ containing~$\mu$.
Note that $-\mathscr{I}$ is the minimal upper semi-continuous function which is greater than or equal to
$\mathscr{F}$:
if $\mathscr{G}\colon\mathcal M \to [-\infty, 0]$ is upper semi-continuous
and $\mathscr{G}(\mu)\geq \mathscr{F}(\mu)$ holds for every $\mu\in\mathcal M$, then
$\mathscr{G}(\mu)\geq -\mathscr{I}(\mu)$ holds for every $\mu\in\mathcal M$.
The simplest example in which $-\mathscr{I}\neq \mathscr{F}$ holds is the quadratic map $f(x)=4x(1-x)$.
We have $-\mathscr{I}(\delta_0)=-\log2$ and $\mathscr{F}(\delta_0)=-\log4$.

\begin{theorema}
Let $f\colon X\to X$ be a topologically exact $S$-unimodal map with non-recurrent flat critical point
that is
of polynomial order.
Then the Large Deviation Principle holds. The rate function is given by
$\mathscr{I}$.
\end{theorema}

  In the area of one-dimensional dynamics, 
all critical points are often assumed to be non-flat.
Otherwise, interactions between the contraction ruled by the critical point
and the expansion away from the critical point become more delicate.
Flat critical points behave like neutral fixed points by trapping nearby orbits for a very long period of time,
and hence can influence on statistical properties of the map.

By the result of Benedicks and Misiurewicz \cite{BenMis89},
for a map as in Theorem A there exists
a $\sigma$-finite invariant measure that is absolutely continuous with respect to the Lebesgue measure. This measure is 
unique up to a multiplicative constant, and is a finite measure if and only if $\int\log|Df(x)|dx>-\infty$.
 If finite, then its normalization is denoted by $\mu_{\rm ac}$ and called {\it an acip}.
Many of the statistical properties of $f$ depend on whether the map has an acip or not, see
Zweim\"uller \cite{Zwe04}. 
Theorem A indicates that the LDP is a special limit theorem which holds regardless of whether
the map has an acip or not.

A proof of Theorem A is briefly outlined as follows. 
In establishing the LDP for one-dimensional non-hypebolic systems, the lower bound is already known to hold
for a broad class of smooth interval maps including those in Theorem A, see
 \cite[Section 7]{Chu11} and \cite[Proposition 4.1]{ChuRivTak}. Hence we do not repeat a proof of it here.
On the other hand, the upper bound is much harder to prove. 
A strategy, developed in \cite{Chu11} and was then carried out successfully in  \cite{ChuRivTak,ChuTak14},
is to construct a ``good'' horseshoe.
We take the same strategy, and our main tool is {\it an inducing scheme equipped 
with a specification-like property} described in Sect.2.2.



The class of maps treated in this paper is disjoint from those treated
in \cite{ChuRivTak}. In \cite{ChuRivTak} all critical points are assumed to be non-flat,
and in the construction of good horseshoes 
 the following estimate was used in order to evaluate the effect of each return to a critical zone (See \cite[Lemma 3.2]{ChuRivTak}):
for every interval $\widehat U$ contained in a small neighborhood of the critical set and every subinterval $U$ of $\widehat U$,  
\begin{equation*} \frac{|f(U)|}{ |f(\widehat U)|}
\leq C_0\frac{ |U| }{ |\widehat U|}, \end{equation*}
where $C_0>1$ is a uniform constant.
This estimate obviously fails in a neighborhood of a flat critical point.
 The inducing scheme equipped with a specification-like property 
enables us to dispense with this type of estimate.
Together with the assumption of non-recurrence,
we use the assumption on the flatness of the critical point solely for constructing this inducing scheme,
see Lemma \ref{keyest}.


We now state a corollary which 
is a direct consequence of Theorem A and of the Contraction Principle \cite{DemZei98}.
For each continuous function $\phi\colon X\to\mathbb R$
and an integer $n\geq1$ write
$ S_n\phi
=
\sum_{i=0}^{n-1}\phi\circ f^i$ and put
$$ c_\phi
=
\inf_{x\in X}\liminf_{n\to\infty}\frac{1}{n}S_n\phi(x)
\quad\text{and}\quad
d_\phi
=
\sup_{x\in X}\limsup_{n\to\infty}\frac{1}{n}S_n\phi(x).$$
Define $q_\phi\colon \mathbb R\mapsto [0, + \infty]$ by
\begin{displaymath}
q_\phi(s)
=
\inf\left\{\mathscr{I}(\mu)\colon\mu\in\mathcal M, \int\phi d\mu=s\right\}.
\end{displaymath}
This function is bounded on~$[c_\phi, d_\phi]$ and constant equal to~$+ \infty$ on~$\mathbb R \setminus [c_\phi,d_\phi]$.
Moreover, $q_\phi$ is convex on~$\mathbb R$, and therefore continuous on~$(c_\phi,d_\phi)$.
\begin{cor1}
Let $f\colon X\to X$ be a topologically exact $S$-unimodal map with non-recurrent flat critical point
that is
of polynomial order.
 For every continuous function $\phi\colon X\to\mathbb R$ satisfying $c_\phi<d_\phi$ and for every interval~$K$ intersecting $(c_\phi,d_\phi)$,
$$ \lim_{n\to\infty}\frac{1}{n}\log \left|\left\{x\in X\colon \frac{1}{n}S_n\phi(x) \in K\right\}\right|
=
-\inf_{s\in K}q_\phi(s).$$
\end{cor1}

A ``local'' version of this type of formula was proved by Keller and Nowicki \cite{KelNow92},
Melbourne and Nicol \cite{MelNic08}, Rey-Bellet and Young \cite{ReyYou08} under
assumptions of some (weak form of) hyperbolicity. They necessarily imply the existence of 
 acips,
and the interval $K$ is required to be sufficiently close to the corresponding empirical mean.

It is important to know for which $\phi$ and $K$ the convergence of the Lebesgue measure
of the set is exponential.
 The method of Melbourne and Nicol \cite{MelNic08} is applicable
 to the case where $f$ has an acip and $K$ is sufficiently close to the empirical mean,
  and yields a sub-exponential upper bound on the Lebesgue measure 
 of the set in the Corollary. This bound cannot be improved with their method, 
 because their bound is closely linked to the decay rate of the tail probability of the 
  associated inducing scheme.
The decay rate for the map $f$ as in Theorem A
 is sub-exponential, see \cite[Proposition 1]{Zwe04}.
A characterization of the zeros of $q_\phi$ would allow us to establish an exponential convergence.

 In this way we are led to the analysis of zeros of the rate function \eqref{rate}.
A measure $\mu\in\mathcal M$ is {\it a post-critical measure}
if there exists an increasing sequence $\{m_i\}_{i\geq0}$ of positive integers such that
$\delta_{c}^{m_i}$ converges weakly to $\mu$ as $i\to\infty$.
Since $\mathcal M$ is compact, post-critical measures exist.
Each post-critical measure is $f$-invariant, and its support is contained in $\omega(c)$.




\begin{theoremb}\label{thmb}
Let $f\colon X\to X$ be a topologically exact $S$-unimodal map with non-recurrent flat critical point
that is of polynomial order. 
 If $\mu\in\mathcal M(f)$ is a post-critical measure,
then $\mathscr{I}(\mu)=0$.
\end{theoremb}

The next theorem gives a partial characterization of the zeros of the rate function.
\begin{theoremc}\label{thmc}
Let $f\colon X\to X$ be a topologically exact $S$-unimodal map with a non-recurrent flat critical point $c$
that is of polynomial order. Then the following holds:
\begin{itemize}

\item[(i)] Assume $f$ has an acip, $f|_{\omega(c)}$ is uniquely ergodic,
 and the unique post-critical measure denoted by $\delta(c)$ has zero entropy.
 Then $$\{\mu\in\mathcal M(f)\colon\mathscr{I}(\mu)=0\}=\{p\delta(c)+(1-p)\mu_{\rm ac}\colon 0\leq p\leq1\};$$

\item[(ii)] Assume $f$ has no acip. If
$\mu\in\mathcal M(f)$ and $\mathscr{I}(\mu)=0$, then $\mu(\omega(c))=1$.
\end{itemize}
\end{theoremc}



To illustrate our main results, consider a parametrized family $\{f_b\}_{b>0}$  
of unimodal maps given by 
\begin{equation}\label{fb}f_b(x)=\begin{cases}
-2^{2^b}\left|x-1/2\right|^{\left|x-1/2\right|^{-b}}+1&\text{ for }x\in[0,1]\setminus\{1/2\};\\
1&\text{ for }x=1/2.
\end{cases}\end{equation}
The $1/2$ is a flat critical point of polynomial order.
A tedious computation shows that $f_b$ has negative Schwarzian derivative, for example,
 for every $b\geq1/\sqrt{6}$.
Note that $f_b(0)=0=f_b(1)$. The Minimum Principle \cite{dMevSt93} implies $Df_b(0)>1$.
Then, from Singer's Theorem \cite{Sin78}
all periodic points are hyperbolic repelling. 
Hence $f_b$ is topologically conjugate to the full tent map and so is topologically exact.
By Theorem A, the LDP holds.
Since $\int\log|Df_b(x)|dx>-\infty$ holds if and only if $b<1$, $f_b$ has an acip (denoted by $\mu_{{\rm ac},b}$) if and only if $b<1$.
The Lebesgue typical behavior changes at $b=1$:
\begin{itemize}
\item[-] for $1/\sqrt{6}\leq b<1$, the measure 
$\delta_{x,b}^n=(1/n)\sum_{i=0}^{n-1}\delta_{f_b^{i}(x)}$
converges weakly as $n\to\infty$ to $\mu_{{\rm ac},b}$ for Lebesgue a.e. $x\in X$;
 
\item[-] for $b\geq 1$, 
$\delta_{x,b}^n$ converges weakly as $n\to\infty$ to the Dirac measure $\delta_0$ at $0$
 for Lebesgue a.e. $x\in X$. 
\end{itemize}

Theorem B and Theorem C together yield a complete characterization of the zeros of the rate function for $f_b$:

\begin{itemize}
\item[-] for $1/\sqrt{6}\leq b<1$, $\mathscr{I}(f_b;\mu)=0$ if and only if there exists
$p\in[0,1]$ such that $\mu=p\delta_0+(1-p)\mu_{{\rm ac},b}$;
 
\item[-] for $b\geq 1$, $\mathscr{I}(f_b;\mu)=0$ if and only if $\mu=\delta_0$. 
\end{itemize}

Let $\phi\colon X\to\mathbb R$ be a continuous function such that $c_\phi<d_\phi$ holds for every $f_b$, $b>0$.
We obtain a complete characterization of the zeros 
of $q_\phi$:

 \begin{itemize}
\item[-] for $1/\sqrt{6}\leq b<1$, 
$q_\phi(s)=0$ holds if and only if
there exists
$p\in[0,1]$ such that $s=p\phi(0)+(1-p)\int\phi d\mu_{{\rm ac},b}$;

\item[-] for $b\geq 1$,
$q_\phi(s)=0$ holds if and only if 
$s=\phi(0)$.  \end{itemize}

In this way, the structure of the set of zeros of the rate function changes at $b=1$.
This type of qualitative changes is well-known in the context of probability and statistical mechanics,
notably in the large deviations for the Curie-Weiss model (See e.g. Ellis \cite{Ell85}, Rassoul-Agha and Sepp\"al\"ainen \cite{RasSep}).

For each $b\in[1/\sqrt{6},1)$ let $p_b^+\in X$ denote the orientation-reversing fixed point of $f_b$, and 
$p_b^-$ the preimage of $p_b^+$ by $f_b$ which is not $p_b^+$.
The first return map to the interval $(p_b^-,p_b^+)$ defines an inducing scheme to which the acip 
$\mu_{{\rm ac},b}$ lifts. From the result of Zweim\"uller \cite{Zwe04}, this inducing scheme has polynomial tail with 
respect to the Lebesgue measure, uniformly over all $b$ contained in each compact subinterval 
of $[1/\sqrt{6},1)$.
Then, the result of Freitas and Todd \cite{FreTod09} on statistical stability implies that
$b\in[1/\sqrt{6},1)\mapsto\mu_{{\rm ac},b}\in\mathcal M$ is continuous (continuous in the $L^1$ norm). 
In Proposition \ref{converge} we show that $\mu_{{\rm ac},b}$ converges weakly to $\delta_0$ as $b\nearrow1$.
As a consequence, the set of zeros of the rate function for $f_b$
 depends continuously on $b>0$
(See FIGURE 1).
This type of changes in the rate functions also occur for the Manneville-Pomeau maps. For details,
see Appendix B. 

\begin{figure}
\begin{center}
\includegraphics[height=5cm,width=6.5cm]{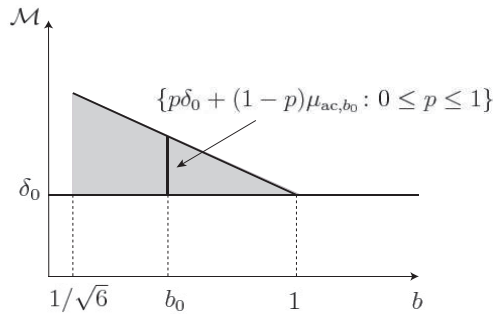}
\caption{The sets of zeros of the rate functions for the family $\{f_b\}_{b>0}$.}
\end{center}
\end{figure}

We point out 
one key difference between non-flat and flat critical points appearing
 in the rate functions.
Let $f$ be an $S$-unimodal map with a critical point $c$. We say $c$ is {\it non-flat} if 
there exist a constant $\ell_c>1$ and 
$C^3$ diffeomorphisms $\phi$, $\psi$ of $\mathbb R$ such that
$\phi(c)=0=\psi(f(c))$ and $|\phi(x)|^{\ell_c}=\psi(f(x))$ for all $x$ near $c$.
An $S$-unimodal map $f$ with a non-flat critical point $c$ satisfies {\it the Collet-Eckmann condition} \cite{ColEck83} if
$$\liminf_{n\to\infty}\frac{1}{n}\log|Df^n(f(c))|>0.$$
This condition implies the existence of an acip \cite{dMevSt93}. 
This measure is unique and also denoted by $\mu_{\rm ac}$.
For a topologically exact $S$-unimodal map with non-flat critical point and satisfying the Collet-Eckmann
condition \cite{ColEck83}, the LDP holds \cite{ChuRivTak}
and the corresponding rate function vanishes only at the acip (See Appendix A for details).

\begin{thm}{\rm (Keller and Nowicki \cite[Theorem 1.2]{KelNow92})}\label{kelnow}
Let $f\colon X\to X$ be an $S$-unimodal map with non-flat critical point satisfying the Collet-Eckmann condition.
Let $\phi\colon X\to\mathbb R$ be of bounded variation such that
$\sigma_\phi^2>0$, 
 where $$\sigma_\phi^2={\rm Var}(\phi)+2\sum_{n=1}^\infty{\rm Cov}(\phi,\phi\circ f^n)$$
 and {\rm Var} and {\rm Cov} denote variance and covariance respectively. 
Then, for sufficiently small $\epsilon>0$ we have
\begin{equation*}\label{kel}
\lim_{n\to\infty}\frac{1}{n}\log \left|\left\{ x\in X\colon \left |\frac{1}{n}S_n\phi(x) - \int \phi d\mu_{\rm ac}\right| > \epsilon   \right\}\right|<0.
\end{equation*}
\end{thm}

The exponential convergence in Theorem \ref{kelnow} no
longer holds for a map in Theorem A having an acip.
Indeed, for such a map $f$,
by Theorem B
the rate function vanishes at each post-critical measure.
Let $\mu\in\mathcal M(f)$ be a post-critical measure and $\phi\colon X\to\mathbb R$
be continuous such that $\int\phi\mu_{\rm ac}\neq\int\phi d\mu$.
Since the rate function is convex, $q_\phi(s)=0$ holds
for every $s\in\{p\int\phi d\mu+(1-p)\int\phi d\mu_{\rm ac}\colon p\in[0,1]\}$.
From this and the Corollary there exists
$\epsilon>0$ such that
$$\lim_{n\to\infty}\frac{1}{n}\log \left|\left\{x\in X\colon \left |\frac{1}{n}S_n\phi(x) - \int\phi d\mu_{\rm ac}\right| > \epsilon  \right\}\right|=0.$$

The rest of this paper consists of three sections and two appendices.
Sect.2 and Sect.3 are entirely dedicated to a proof of the upper bound \eqref{up}.
In Sect.2 we introduce inducing schemes and a specification-like property associated with them.
 We then show that this property is indeed satisfied for the map in Theorem A and
   for an inducing scheme obtained from the first return map to a
properly chosen small interval containing the flat critical point.
Building on these, in Sect.3 we construct good horseshoes 
and complete a proof of the upper bound in Theorem A.
Theorem B and Theorem C are proved in Sect.4.
In Appendix A we analyze the rate function of a map satisfying the Collet-Eckmann condition.
In Appendix B we treat the LDP for intermittent maps of the interval.

\section{Preliminaries for upper bound}
In this section we do preliminary works for obtaining the upper bound.
In Sect.\ref{ischeme} we introduce inducing schemes and describe their basic properties.
In Sect.\ref{special} we construct an inducing scheme with the specification-like property.

\subsection{Inducing schemes}\label{ischeme}
Let $f$ be a unimodal map. 
Let $U$ be an interval of $X$ and $n\ge 1$ an integer.
Each connected component of $f^{-n}(U)$ is called a {\it pull-back} of $U$ by $f^n$.
A pull-back $J$ of $U$ by $f^n$ is {\it diffeomorphic} if $f^n: J\to U$ is a diffeomorphism.
An open subinterval $I$ of $X$ is {\it nice} if $f^n(\partial I)\cap I=\emptyset$ holds for every $n\geq1$.

Assume the critical point $c$ of $f$ is non-recurrent.
Let $I$ be a nice interval which contains $c$
and satisfies $\overline{\{f^n(c)\colon n\geq1\}}\cap\overline{ I}=\emptyset$.
Diffeomorphic pull-backs of $I$ are mutually disjoint, and
every pull-back of $I$ is diffeomorphic. 
If $W$ is a pull-back of $I$, 
the integer $r\geq1$
such that $f^r(W)=I$ is unique.
 This $r=r(W)$ is called {\it an inducing time of $W$}.
The pull-back $W$ of $I$ is {\it primitive} if 
$f^k(W)\cap I=\emptyset$ holds for every $k\in\{0,\ldots, r-1\}\setminus\{0\}$.

{\it The first return time to $I$} is a function $R\colon X\to \mathbb Z_{>0}\cup\{\infty\}$ 
defined by
$$R(x)=\inf\left(\{n\geq1\colon f^n(x)\in I\}\cup\{\infty\}\right).$$
 If $W$ is a primitive pull-back of $I$, then $R$ is constant on $W$ and this common value is denoted by $R(W)$.
Let $\mathcal W$ denote the collection of all primitive pull-backs of $I$ which are contained in $I$.
The triplet $(I,\mathcal{W},R)$ is called
{\it an inducing scheme}.
Define {\it an induced map} $\widehat f:\cup_{J\in\mathcal W} J\to I$ by $\widehat f(x)=f^{R(J_x)}(x)$
where $J_x$ is the element of $\mathcal W$ containing $x$.
 
Let $f$ be a topologically transitive $S$-unimodal map with non-recurrent critical point.
Then any inducing scheme $(I,\mathcal{W},R)$ satisfies the following properties:
\medskip

\noindent{\underline{\it Expansion property:}}
there exist $\lambda>1$ and an integer $m\geq1$ such that 
\begin{equation*}
|D(\widehat f)^m(x)|\geq\lambda\quad \text{for all $x\in \bigcap_{n=0}^{m-1} (\widehat f)^{-n}\left(\bigcup_{J\in\mathcal W} J\right)$}.\end{equation*}

\noindent{\underline{\it Bounded distortion:}} there exists $C>0$ such that
for all  $x, y$ which are contained in the same element of $\mathcal W$,
$$\log\frac{|D\widehat f(x)|}{ |D\widehat f(y)|}\leq C|\widehat f(x)-\widehat f(y)|.$$

 \subsubsection*{\underline{Liftability}}
Consider the dynamical system
 on $\bigcap_{n\geq0}({\widehat f})^{-n}\left(\bigcup_{J\in\mathcal W} J\right)$ generated by $\widehat f$, and
 let $\mathcal M(\widehat f)$ denote the set of $\widehat f$-invariant Borel probability measures.
For a measure $\widehat\mu\in\mathcal M(\widehat f)$ 
for which $\int Rd\widehat\mu$ is finite,
define $$\mathcal L(\widehat\mu)=\frac{1}{\int R d\widehat\mu}\sum_{J\in\mathcal W}\sum_{n=0}^{R(J)-1}(f^n)_*(\widehat \mu|_{J}).$$
It is straightforward to check that 
$\mathcal L(\widehat\mu)\in\mathcal M(f)$. 
A measure $\mu\in\mathcal M(f)$ is {\it liftable to the inducing scheme $(I,\mathcal W,R)$} if
there exists $\widehat\mu\in\mathcal M(\widehat f)$ such that $\int R d\widehat\mu$ is finite and $\mathcal L(\widehat\mu)=\mu$.
Not all measures are liftable. For instance measures whose supports are contained in $\omega(c)$
are not liftable.

\begin{lemma}\label{liftlem}
Let $f$ be a topologically transitive $S$-unimodal map with non-recurrent critical point,
and let  $(I,\mathcal W,R)$ be an inducing scheme.
If $\mu\in\mathcal M(f)$ and $\mu(I)>0$, then $\mu$ is liftable to $(I,\mathcal W,R)$.
\end{lemma}
\begin{proof}
Since $R$ is the first return time to $I$, 
if $\mu\in\mathcal M(f)$ and $\mu(I)>0$, then $\int Rd\mu$ is finite (in fact, equal to $1$, see Kac \cite{Kac47}).
 From the result of Zweim\"uller \cite{Zwe05}, $\mu$ is liftable.
\end{proof}

\noindent 
Moreover, $\overline{I}=\overline{\bigcup_{J\in\mathcal W}J}$ holds.
For maps with non-flat critical point, these are known as a folklore, and 
they also hold for maps with flat critical point.

\begin{lemma}\label{positive}
Let $f$ be an $S$-unimodal map with non-recurrent critical point.
 Then $\chi(\mu)>0$ holds for every $\mu\in\mathcal M(f)$.
\end{lemma}
\begin{proof}
From Ma\~n\'e's hyperbolicity theorem \cite[Theorem A]{Man85}
and Birkhoff's ergodic theorem,
$\chi(\mu)>0$ holds for every ergodic $\mu\in\mathcal M(f)$ whose support 
does not contain $c$.
Let $\mu\in\mathcal M(f)$ be ergodic whose support contains $c$.
Poincar\'e's recurrence theorem implies $\mu(I)>0$, and hence $\mu$ is liftable. From \cite[Theorem 3]{BruTod09}, $\chi(\mu)>0$ holds.
From the ergodic decomposition theorem, the positivity also holds
for non-ergodic measures.
\end{proof}

\subsection{Specification-like property of the inducing scheme}\label{special}
Specification (See e.g. Young \cite{You90} for the definition) allows us to glue arbitrary orbit segments
together to form one orbit. By the specification-like property of an inducing scheme
$(I,\mathcal W,R)$ we roughly mean a property which allows us to glue orbits 
of part of the tail set $\{R>n\}=\{x\in I\colon R(x)>n\}$
to the nice interval $I$ to form a pull-back of $I$ whose first return time is roughly equal to $n$.
We additionally request that the size of this pull-back is not too small.
The next proposition asserts the existence of an inducing scheme with the specification-like property.

\begin{prop} \label{nse}
Let $f$ be a topologically transitive $S$-unimodal map with non-recurrent flat critical point
that is of polynomial order. 
There exists an inducing scheme $(I,\mathcal W,R)$ with the following property:
for every $\varepsilon >0$ there exist $C(\varepsilon)>0$ and $n_0\geq1$ such that
for every integer $n\geq n_0$ and every connected component $A$ of $\{R>n\}$,
there exists $J\in\mathcal W$ which is contained in $A$ and satisfies 
\begin{equation*}\label{NS}
n<R(J)\leq(1+\varepsilon)n\quad\text{ and}\quad | J|\geq\frac{C(\varepsilon)}{n}|A|.
\end{equation*}
\end{prop}
We will use Proposition \ref{nse} in the proof of Proposition \ref{upper0} to 
construct a horseshoe whose branches are pull-backs of $I$ with a common inducing time.



\begin{remark}
In the case where $R$ is monotone (i.e., $R(x)\to\infty$ monotonically as $x\to c$), the set $\{R>n\}$
is connected. Then the estimates in 
Proposition \ref{nse}
follow from the result of Zweim\"uller \cite[Proposition 1]{Zwe04}.
However, the monotonicity does not hold in general, even when the critical point is pre-periodic.
\end{remark}

   \begin{proof}[Proof of Proposition \ref{nse}]
   We use the following notation. 
For a point $x^+\in X$ with $x^+>c$ 
  let $x^-$ denote the point in $X\setminus\{x^+\}$ 
such that $f(x^-)=f(x^+)$.
For two subsets $A$, $B$ of $X$, $A<B$ indicates $\sup A\leq\inf B$.
Given a subset $A^+$ of $X$ with $A^+>\{c\}$ 
define $A^-=\{x^-\in X\colon x^+\in A^+\}$.

   The proof of Proposition \ref{nse} consists of four steps.
   We start by choosing an inducing scheme, and then
   analyze the distribution of the first return time.
  In these two steps the assumption on the order of the flat critical point is never used.
  In the third step, we prove a key estimate (Lemma \ref{keyest}) associated with the inducing scheme.
  In the last step, we combine the analysis on the distribution of the first return time with the key estimate
  and establish the specification-like property.
  \medskip
    
\noindent{\it Step 1: \underline{Choice of inducing scheme.}}  
Since $f$ is topologically transitive and has a periodic point of odd period different from $1$, it is topologically mixing
\cite[Theorem 2.20]{Rue17}. Hence
       $c$ is accumulated by periodic points from both sides.
There exist a nice interval $I=(a_0^-,a_0^+)$ 
       which satisfies $\overline{I}\cap\overline{\{f^n(c)\colon n\geq1\}}=\emptyset$, and
         $f^{R_0}(a_0^-)=a_0^-$ or  $f^{R_0}(a_0^+)=a_0^+$ where $R_0=\min\{n\geq1\colon f^n(I)\cap I\neq\emptyset\}$.     Without loss of generality we may assume $f^{R_0}(a_0^+)=a_0^+$.
Let $\mathcal W$ denote the collection of primitive pull-backs of $I$ and $R$ the first return time to $I$.
In what follows we show that the inducing scheme $(I,\mathcal W,R)$ satisfies the desired properties.
\medskip

\begin{lemma}\label{lemmore}
Let $W_1$, $W_2$ be distinct primitive pull-backs of $I$ such that $R(W_1)=R(W_2)$.
There exists a primitive pull-back $W$ of $I$ such that $W_1<W<W_2$ and $R(W)<R(W_1)$.
\end{lemma}

\begin{proof}
Put $m=R(W_1)=R(W_2)$.
Let $U$ denote the minimal open interval containing $W_1$ and $W_2$.
Let $n\geq1$ be the smallest integer such that $c\in f^n(U)$. We must have $n<m$. 
Since $f^k(W_1\cup W_2)\cap I=\emptyset$
for every $k\in\{1,2,\ldots,n\}$,  $f^k(U)\cap I=\emptyset$
for every $k\in\{1,2,\ldots,n-1\}$.
Since $W_1$, $W_2$ are primitive, $f^n(W_1\cup W_2)\cap I=\emptyset$.
Define $W$ to be the pull-back of $I$ by $f^n$ which is contained in $U$. 
\end{proof}

\noindent{\it Step 2: \underline{Analysis of distribution of the first return time.}}  
Let 
$U$ be a subset of $X$. A subset $W$ of $U$ is {\it the minimal pull-back of $I$ in $U$} 
if it is a primitive pull-back of $I$ and for any other primitive pull-back $W'$ of $I$ which is contained in $U$,
  $R(W')> R(W)$ holds.

Define $V_0^+$ to be the minimal pull-back of $I$ in $(a_0^+,\inf\{f^n(c)\colon n\geq1, f^n(c)>c\})$.
We construct by induction a finite sequence $V_0^+$, $V_1^+,\ldots$ of primitive pull-backs of $I$ as follows.
    Define $V_0^+$ to be the minimal pull-back of $I$ in $(a_0^+,\inf\{f^n(c)\colon n\geq1, f^n(c)>c\})$.
 Let $i\geq0$ be an integer and suppose $V_0^+,\ldots,V_i^+$ has been defined.
From Lemma \ref{lemmore}
 there are two cases: 
 \smallskip
 
 (i) for any primitive pull-back $W$ of $I$ such that $V_i^+<W$,
$R(V_i^+)<R(W)$; 

(ii) there exists a primitive pull-back $W$ of $I$ such that $V_i^+<W$ and $R(V_i^+)>R(W)$.
\smallskip

In case (i) we stop the construction. In case (ii) 
write $V_i^+=(a,b)$ and define $V_{i+1}^+$ to be the minimal pull-back of $I$ 
in $(b,1)$. 
This construction stops in finite time and we end up with
a sequence $\{V_i^+\}_{i=0}^{N}$ of primitive pull-backs of $I$
with the following properties:

\begin{itemize}
\item[1.] $I<V_0^+<V_1^+<\cdots<V_N^+$; 

\item[2.] $R(V_0^+)>R(V_1^+)>\cdots>R(V_{N}^+)=1$;

\item[3.] if $\partial V_0^+\cap\partial I=\emptyset$ and  $W$ is a primitive pull-back of $I$ such that
 $I<W<V_0^+$, then $R(W)>R(V_0^+)$;

\item[4.] If $W$ is a primitive pull-back of $I$ such that $V_{i}^+<W<V_{i+1}^+$ holds
for some $i\in\{0,\ldots, N-1\}$, then $R(W)> R(V_{i}^+)$. 
\end{itemize}
\begin{remark}\label{chebyshev}
{\rm In the case $f(c)=1$ and $f(0)=0=f(1)$ like the Chebyshev quadratic $f(x)=4x(1-x)$, 
we choose $I=(a_0^-,a_0^+)$ where $a_0^+$ is the fixed point which is not $0$.
Let $b^-<c$ be such that $f(b^-)=a_0^-$. We have
$V_0^-=(b^-,a_0^-)$, $N=0$ and
 $R_k=k+2$ for every $k\geq0$.
In particular, for every $k\geq0$, $J_k^\pm$ are primitive pull-backs of $I$ with $R(J_k^\pm)=k+2$.}
\end{remark}

We define
a sequence $\{J_k^+\}_{k\geq0}$ of open subintervals of $I$,
 a sequence $\{R_k\}_{k\geq0}$ of positive integers inductively as follows.
Define $J_0^+$ to be the pull-back of $I$ by $f^{R_0}$ which is contained in $I$ and 
satisfies $\{c\}<J_0^+$. 
Now, let $k\geq0$ and suppose we have defined
$J_0^+,J_1^+,\ldots,J_{k}^+$ and $R_0,R_1,\ldots,R_k$
with the desired properties 
such that the set
$ f^{R_k}\left(I\setminus(\bigcup_{0\leq n\leq k}  J_n^+)\right)$
does not intersect $I$. Since this set contains $V_0^+$ or $V_0^-$
depending on whether $f^{R_k}(c)>c$ or $f^{R_k}(c)<c$,
 the minimal pull-back of $I$ in this set, denoted by
$W$, belongs to $\{V_{i}^-\}_{0\leq i\leq N}\cup\{V_{i}^+\}_{0\leq i\leq N}$.
Let $J$ denote the pull-back of $I$ by $f^{R_k+R(W)}$ which is 
contained in $I\setminus(\bigcup_{0\leq n\leq k} J_n^-\cup J_n^+)$ and
satisfies $\{c\}<J$. 
If $\partial J_k^+\cap\partial J\neq\emptyset$, then set $J_{k+1}^+=J$
and $R_{k+1}=R_k+R(W)$.
Otherwise, set $J_{k+2}^+=J$ and
define $J_{k+1}^+$ to be the maximal open interval sandwiched by $J_{k}^+$ and $J_{k+2}^+$.
Set $R_{i}=R_{k}+R(W)$ for $i\in\{k+1,k+2\}$.

From the definition the following holds:

\begin{itemize}

\item[1.] $\{c\}<\cdots<J_{k+1}^+<J_k^+<\cdots<J_0^+$;



\item[2.] 
for each $k\geq0$ one of the following holds:

\begin{itemize}
\item[(i)] $J_{k}^+\in \mathcal W$ 
and $R(J_k^+)=R_k$;

\item[(ii)]  $J_{k}^+\notin \mathcal W$,
 $I\cap f^i(J_{k}^+)=\emptyset$ for every $i\in\{1,\ldots,R_{k}\}$,
$\partial I\cap\partial f^{R_{k}}(J_{k}^+)\neq\emptyset$;

\end{itemize}

\item[3.] If $J_{k+1}^+\notin\mathcal W$, then $J_{k}^+\in\mathcal W$
and $R_{k}=R_{k+1}$.
\end{itemize}
It follows that for every $k\geq0$,
\begin{equation}\label{R1}0\leq R_{k+1}-R_{k}\leq R(V_0^+)\quad\text{and}\quad
\frac{k}{2}\leq R_k-R_0\leq R(V_0^+)k.\end{equation}

In order to control distortions,
fix $\tau\in(0,1)$ such that 
 $\overline{\{f^n(c)\colon n\geq1\}}$ does not intersect the concentric closed interval with $I$ of length $(1+2\tau)|I|$,
and for each $k\geq0$
there exists a subinterval of $X$ on which
$f^{R_k}$ is a diffeomorphism and the image
 contains the concentric open interval
with $f^{R_k}(J_k^+\cup J_{k+1}^+)$ of length $(1+2\tau)|f^{R_k}(J_k^+\cup J_{k+1}^+)|$.
The latter condition is realized for sufficiently small $\tau$ because
 $\overline{\{f^n(c)\colon n\geq1\}}$ does not intersect 
$\bigcup_{i=0}^NV_i^-\cup I\cup\bigcup_{i=0}^N V_i^+$.
Put $K_\tau=(\tau/(1+\tau))^2$.

\begin{lemma}\label{LemJK}
There are constants $0<\theta_0 <\theta_1<1$ such that  for every $k\geq0$,
\begin{equation*}\label{disJK}
\theta_0\leq\frac{ | J_k^+ | }{ |J_k^+ \cup J_{k+1}^+ | }\leq \theta_1.\end{equation*}
\end{lemma}

\begin{proof}
Put $$\Delta=\begin{cases}\inf\{|x-y|\colon x\in I, y\in V_0^-\cup V_0^+\}&\text{if $\partial I\cap \partial V_0^+=\emptyset$;}\\
\min\{|V_0^-|,|V_0^+|\}&\text{otherwise.}\end{cases}$$
We treat three cases separately.
If $J_k^+\in\mathcal W$ and $J_{k+1}^+\in\mathcal W$, then
$f^{R_k}(J_k^+)=I$ and $f^{R_k}(J_{k+1}^+)\in\{V_0^-,V_0^+\}$. 
By the choice of $\tau$ and the Koebe Principle  \cite[p.277 Theorem 1.2]{dMevSt93},
$$\dfrac{| J_k^+ |}{ |J_k^+\cup J_{k+1}^+| } \ge K_\tau \dfrac{| f^{R_k } (J_k^+) |}{ |f^{R_k } (J_k^+\cup J_{k+1}^+ )|  } 
\ge K_\tau |I|$$
and
$$\dfrac{| J_{k+1}^+ |}{ |J_k^+\cup J_{k+1}^+| } \ge K_\tau \dfrac{| f^{R_k } (J_{k+1}^+) |}{ |f^{R_k } (J_k^+\cup J_{k+1}^+ )|  } 
\ge K_\tau\Delta.$$
If $J_k^+\in\mathcal W$ and $J_{k+1}^+\notin\mathcal W$,
then
$f^{R_k}(J_k^+)=I$, there exists $i\in\{0,\ldots,N\}$ such that $f^{R_k}(J_{k+2}^+)\in\{V_i^-,V_i^+\}$ and
the interval $f^{R_k}(J_{k+1}^+)$ is the one sandwiched by $f^{R_k}(J_{k+2}^+)$ and $I$, 
and hence $|f^{R_k}(J_{k+1}^+)|\geq  \Delta.$
By the Koebe Principle again we obtain the same set of inequalities.
If $J_k^+\notin\mathcal W$ and $J_{k+1}^+\in\mathcal W$, then
there exists $i\in\{0,\ldots,N\}$ such that $f^{R_k}(J_{k+1}^+)\in\{V_i^-,V_i^+\}$,
and the interval $f^{R_k}(J_{k}^+)$ is the one sandwiched by $f^{R_k}(J_{k+1}^+)$ and $I$, 
and hence $|f^{R_k}(J_{k}^+)|\geq  \Delta.$
By the Koebe Principle,
$$\dfrac{| J_k^+ |}{ |J_k^+\cup J_{k+1}^+| } \ge K_\tau\dfrac{| f^{R_k } (J_k^+) |}{ |f^{R_k } (J_k^+\cup J_{k+1}^+ )|  } 
\ge K_\tau  \Delta$$
and
$$\dfrac{| J_{k+1}^+ |}{ |J_k^+\cup J_{k+1}^+| } \geq K_\tau \dfrac{| f^{R_k } (J_{k+1}^+) |}{ |f^{R_k } (J_k^+\cup J_{k+1}^+ )|  } 
\ge K_\tau \min\{|V_i^-|,|V_i^+|\}.$$
Put $\theta_0=K_\tau\min\{|V_0^-|,|I|,|V_0^+|,\Delta\}$
and $\theta_1=1-K_\tau\min_{1\leq i\leq N}\min\{|V_i^-|,|V_i^+|,\Delta\}$.
\end{proof}

\noindent{\it Step 3: \underline{Key estimates.}} 
For each $k\geq0$ put  $J_k^+=(a_{k+1}^+,a_k^+)$, 
$I_k^+=(c,a_k^+)$ and $I_k=(a_k^-,a_k^+)$.
To obtain the large deviation upper bound,
one key estimate is that of the speeds of decay of $|I_k^\pm|$ as $k\to\infty$.
For a non-flat critical point it is not hard to show that 
$$\limsup_{k\to\infty } \frac{|I_{k+1}^+ |}{|I_k^+ |}  < 1  \quad \text{ and } \quad
\limsup_{k\to\infty } \frac{|I_{k+1}^- |}{|I_k^- |}  < 1.$$ 
These estimates can be used to show that the inducing scheme 
is of {\it bounded slope},  which implies the desired upper bound \eqref{up}, see
\cite{Chu11}. 
Since these estimates no longer hold for a flat critical point,
we now prove slightly different estimates and work with them.

Let $[ $ $\cdot $ $]$ denote the integer part.
The next lemma is a key estimate for establishing the specification-like property of the inducing scheme
$(I,\mathcal W,R)$. 
As far as the proof of Theorem A is concerned, this is the only place where the order of flatness of the critical point
comes into play.
\begin{lemma}\label{keyest}
Assume the critical point is of polynomial order.
For every $\varepsilon >0$ we have
$$\limsup_{k\to\infty } \frac{|I_{[(1+\varepsilon)k]}^+ |}{|I_k^+ |}  < 1  \quad \text{ and } \quad
\limsup_{k\to\infty } \frac{|I_{[(1+\varepsilon)k]}^- |}{|I_k^- |}  < 1.$$ 
\end{lemma}

\begin{proof}
From Ma\~n\'e's hyperbolicity theorem \cite[Theorem A]{Man85}
there are constants $C_1>0$ and $\lambda>1$ such that for every $n\geq1$ and every $x\in X$
such that $x,f(x),\ldots,f^{n-1}(x)\notin I$, 
\begin{equation}\label{mane}|Df^n(x) | \ge C_1\lambda^{n}.\end{equation}
Put $D_k=|Df^{R_k}(f(c))|$.
There exists $C>1$ such that for every $k\geq0$,
\begin{equation}\label{once}C^{-1}\leq|I_k^+|^{\ell(a_k^+)}D_k\leq C.\end{equation}
Put $k'=[(1+\varepsilon)k]$.


Taking logs of \eqref{once} gives
\begin{equation}\label{A1eq-1}\log|I_k^+|\approx-\frac{\log D_k}{\ell(a_k^+)},\end{equation}
where $\approx$ indicates that the ratio of the two numbers goes to $1$ as $k\to\infty$.
Then
\begin{equation}\label{A1eq0}
\frac{\log|I_{k'}^+|}{\log|I_{k}^+|}\approx\frac{\log D_{k'}}{ \log D_k}\cdot\frac{\ell(a_k^+)}{\ell(a_{k'}^+)}.
\end{equation}
If $\log|I_{k'}^+|<(1+\varepsilon^2)\log|I_{k}^+|$, then
using \eqref{A1eq-1},
\begin{equation}\label{A1eq1}
\frac{|I_{k'}^+|}{|I_{k}^+|}\leq\exp\left(\varepsilon^2\log|I_k^+|\right)
\leq\exp\left(-\varepsilon^2(1-\varepsilon)\frac{\log D_k}{\ell(a_k^+)}\right),\end{equation}
which goes to $0$ as $k\to\infty$ because $|I_k^+|\to0$ and $k/\ell(a_k^+)\to\infty$ from \eqref{A1eq-1}.
If  $\log|I_{k'}^+|\geq(1+\varepsilon^2)\log|I_{k}^+|$, then
from \eqref{A1eq0},
$$\frac{\log D_{k'}}{\log D_{k}}\cdot\frac{\ell(a_k^+)}{\ell(a_{k'}^+)}\leq 1+2\varepsilon^2.$$
From \eqref{mane} there exists $C\in(0,1)$ such that
$$D_{k'}=|Df^{R_{k'}-R_k}(f^{R_k+1}(c))||Df^{R_{k}}(f(c))|\geq CD_k.$$
We have
\begin{equation*}\frac{\log D_k}{\log D_{k'}}\leq \frac{\log D_k}{\log D_k+\log C}
\leq 1+\varepsilon^2.\end{equation*}
Hence
$$\frac{\ell(a_k^+)}{\ell(a_{k'}^+)}\leq(1+2\varepsilon^2) (1+\varepsilon^2)
=\frac{(1+2\varepsilon^2) (1+\varepsilon^2)}{1+\varepsilon}\leq\frac{1+4\varepsilon^2}{1+\varepsilon}.$$
Substituting $\ell(a_k^+)=|I_k^+|^{-v(a_k^+)}$ and
$\ell(a_{k'}^+)=|I_{k'}^+|^{-v(a_{k'}^+)}$
into this inequality gives
\begin{equation}\label{A1eq2}\frac{|I_{k'}^+|}{|I_{k}^+|}\leq
\left(\frac{1+4\varepsilon^2}{1+\varepsilon}\right)^{\frac{1}{\max\{v(x)\colon x\in X\}}}<1,\end{equation}
provided $\varepsilon\in(0,1)$ is sufficiently small.
From \eqref{A1eq1} and \eqref{A1eq2}
 we obtain the first inequality in Lemma \ref{keyest}.
A proof of the second one is analogous.
\end{proof}

\noindent{\it Step 4: \underline{Verification of the specification-like property.}} 
Let $\varepsilon >0$, $n>R(V_0^+)/\varepsilon$ an integer and let $A$ be a connected component of $\{R>n\}$.
We treat two cases separately.
\medskip

\noindent{\it Case I: \underline{$c\notin A$}.} 
Then $A$ is sandwiched by two elements $L_1$, $L_2$ of $\mathcal W$ with
$R(L_1)<R(L_2)\leq n$. Let $i\in\{0,\ldots,N\}$ be the maximal such that
$V_i^+\subset f^{R(L_2)}(A)$ or $V_i^-\subset f^{R(L_2)}(A)$.
Only one of the two inclusions holds and without loss of generality we may assume
$V_i^+\subset f^{R(L_2)}(A)$.
Let $J$ denote the pull-back of $V_i^+$ by $f^{R(L_2)}$ which is contained
 in $A$. We have $J\in\mathcal W$ and $R(J)=R(L_2)+R(V_i^+)\leq n+R(V_i^+)<(1+\varepsilon)n$.
From the choice of $\tau$ in Step 2, there exists a subinterval of $X$ containing $A$ on which 
  $f^{R(L_2)}$ is a diffeomorphism and the image
 contains the concentric open interval with $f^{R(L_2)}(A)$ of length $(1+2\tau)|f^{R(L_2)}(A)|$.
 By the Koebe Principle,
\begin{equation}\label{equa1}\dfrac{|J|}{|A|} \ge K_\tau\dfrac{|f^{R(L_2)}(J)|}{|f^{R(L_2)}(A)|}
\ge K_\tau\min\{|V_i^-|,|V_i^+|\} .\end{equation}
\medskip

\noindent{\it Case II: \underline{$c\in A$}.} 
For each integer $n\geq1$ put 
$$\widehat n=\min\{k>0\colon J_k^+\in\mathcal W, R(J_{k}^+)>n\}.$$
\eqref{R1} implies
$$A\cap \{R\le (1+\varepsilon ) n \}\supset
\bigcup \{J_k^+\colon 
\widehat n\leq k\leq \widehat n+\varepsilon n/R(V_0^+),\ J_k^+\in\mathcal W \}.$$
Either $A=I_{\widehat n}$ or $A=I_{\widehat n-1}$ holds. 
If $A=I_{\widehat n}$, then
Lemma \ref{LemJK} gives
$$\frac{\left|\bigcup \{J_k^+\colon 
\widehat n\leq k\leq \widehat n+\varepsilon n/R(V_0^+),\ J_k^+\in\mathcal W \}\right|}{|I_{\widehat n} \setminus I_{[\widehat n+
\varepsilon n/R(V_0^+)]} |}\geq\min\{\theta_0,1-\theta_1\}.$$
By Lemma \ref{keyest} there exists $\rho=\rho(\varepsilon)\in(0,1)$ such that
$$  \frac{ |I_{\widehat n} \setminus I_{[\widehat n+\varepsilon n/R(V_0^+)
]} | } {|I_{\widehat n}|}    \ge \rho.$$
Putting these two inequalities together we have
\begin{align*}
\frac{\left|\bigcup \{J_k^+\colon 
\widehat n\leq k\leq \widehat n+\varepsilon n/R(V_0^+),\ J_k^+\in\mathcal W \}\right|}{ |I_{\widehat n} |  } \ge \rho \min\{\theta_0,1-\theta_1\}.
 \end{align*}
 Let $k_0$ be an integer in the interval $[\widehat n, \widehat n+\varepsilon n/R(V_0^+)]$
  such that $|J_{k_0}^+|\geq |J_k^+|$ holds for all other integer $k$ in the interval.
Then 
\begin{equation}\label{equa2} \frac{|J_{k_0}^+  |}{|I_{\widehat n}|}\geq\frac{\rho \min\{\theta_0,1-\theta_1\}}{\varepsilon n/R(V_0^+)+1}.\end{equation}
In the case $A=I_{\widehat n-1}$, we argue replacing $\widehat n$ by $\widehat n-1$.
\eqref{equa1} \eqref{equa2} together imply the second estimate in
 Proposition \ref{nse}. \end{proof}

\section{Proof of the upper bound}
In this section we complete the proof of the upper bound \eqref{up} in
three steps. In Sect.\ref{inter}, using the results and constructions in Sect.2 we 
prove intermediate estimates associated with an inducing scheme with the specification-like property. 
In Sect.\ref{exact} we spread this intermediate estimate
to the whole interval $X$ using the topological exactness.
From this overall estimate we derive \eqref{up} in Sect.\ref{end}.

\begin{remark}
In \cite{Chu11} several strong conditions in terms of the distributions of return times
were introduced for non-hyperbolic systems admitting inducing schemes
(or Young's tower), and it was shown that the LDP holds under these conditions. 
The specification-like property in Proposition \ref{nse} is similar to these conditions,
and the contents of this section are mere adaptations of \cite[?]{Chu11} to our setting. 
\end{remark}

\subsection{Intermediate estimate associated with inducing scheme}\label{inter}
For the rest of this section, 
fix an inducing scheme $(I,\mathcal W,R)$ for which the conclusion of Proposition \ref{nse} holds.

\begin{prop}
\label{upper0}
Let $f$ be a topologically exact
$S$-unimodal map with non-recurrent flat critical point that is of polynomial order. 
For every $\varepsilon>0$, an integer $l\geq1$, 
continuous functions
$\phi_1,\ldots,\phi_l\colon X\to\mathbb R$  and $\alpha_1,\ldots,\alpha_l\in\mathbb R$
 there exists $n_{1}\geq1$ with the following property:  for every integer
$n\geq n_{1}$ for which there exists
$x\in I$ such that $(1/n)S_n\phi_i(x)\geq\alpha_i$ for every $i\in\{1,\ldots,l\}$,
there exists $\mu\in\mathcal M(f)$ such that 
$$ \int\phi_id\mu>\alpha_i-\varepsilon\quad\text{ for every }i\in\{1,\ldots,l\},$$
and
$$ \frac{1}{n} \log \left| \left\{x\in I\colon \frac{1}{n}S_n\phi_i(x)\geq\alpha_i\text{ for every }i\in\{1,\ldots,l\}\right\} \right|
\leq
\mathscr{F}(\mu)+\varepsilon. $$
\end{prop}

Before entering a proof of Proposition \ref{upper0} we need a couple of lemmas.
\begin{lemma}\label{back}
Let $f$ be a topologically exact $S$-unimodal map.
For every $\delta>0$ there exists an integer $n_2\geq1$
such that for every integer $n\geq n_2$ and every pull-back $W$ of $I$ by $f^n$
which is contained in $I$,
 $|f^i(W)| \le \delta$ holds for every $i\in\{0,1,\ldots,n-n_2-1\}$.
\end{lemma}
\begin{proof}
Let $\delta>0$. By \cite[Lemma 2.3]{ChuRivTak}, there exists $\eta\in(0,1/2)$ such that for
every integer $n\geq 1$ and every subinterval~$W$ of~$X$ that satisfies $|f^n(W)|\leq\eta$,
 $|f^i(W)| \le \delta$ holds for every $i\in\{0,1,\ldots,n-1\}$.
From Ma\~n\'e's hyperbolicity theorem  \cite[Theorem A]{Man85}
and the expansion property of the inducing scheme,
it is possible to choose an integer $n_2\geq1$ such that if $n>n_2$ and $W$ is a pull-back of $I$ by $f^n$ which
is contained in $I$, then $|f^{n-n_2}(W)|\leq\eta$.
From the property of $\eta$, $|f^i(W)| \le \delta$ holds for every $i\in\{0,1,\ldots,n-n_2-1\}$.
\end{proof}

The next lemma follows from \cite[Lemma 4.5]{ChuRivTak}. See also \cite[Lemma 7]{Chu11}.
 \begin{lemma}\label{horse}
Let $t, q \ge 1$ be integers, and let $L_1,L_2,\ldots,L_t$ be
pull-backs of $I$ by $f^q$ contained in $I$.
Then there exists $\widehat\mu\in\mathcal M(f^q)$ supported on $L_1 \cup
\cdots \cup L_t$ such that the measure $\mu = (1/q) (\widehat\mu +f_*(\widehat\mu)+
\cdots + f_*^{q - 1}(\widehat\mu))$ is in $\mathcal M(f)$ and satisfies
$$ 
\log \left( |L_1| + \cdots + |L_t| \right)\leq q\mathscr{F}(\mu)
+\log\frac{|I|}{K_\tau}.$$
\end{lemma}


\begin{proof}[Proof of Proposition \ref{upper0}]
We define a sequence $\{\widehat{\mathscr D}_n\}_{n\geq0}$ of collections of open subintervals 
of $I$ inductively as follows.
 Start with $\widehat{\mathscr D}_0=\mathcal W$.
 Let $n\geq1$ and suppose $\widehat{\mathscr D}_{n-1}$ has been defined.
 Let $K\in\widehat{\mathscr D}_{n-1}$. If $f^{n}(K)=I$, then define
 $\widehat{\mathscr D}_{n}(K)=\{K'\subset K\colon f^{n}(K')\in \mathcal W\}$.
 Otherwise, define $\widehat{\mathscr D}_{n}(K)=\{K\}$.
 Set $\widehat{\mathscr D}_{n}=\bigcup_{K\in\widehat{\mathscr D}_{n-1}}\widehat{\mathscr D}_{n}(K)$.

Let $\varepsilon>0$, $l\geq1$ an integer,
$\phi_1,\ldots,\phi_l\colon X\to\mathbb R$ be continuous functions and $\alpha_1,\ldots,\alpha_l\in\mathbb R$.
 Let $\mathscr D_n$ denote the collection of $K\in\widehat{\mathscr D}_n$ such that
 there exists $x\in K$ such that  $(1/n)S_n\phi_i(x)\geq \alpha_i$ for every $i\in\{1,\ldots,l\}$.
We show that there exists an integer $n_1\geq1$ such that for every $n\geq n_1$ and
every $K\in\mathscr D_{n}$ there exists a pull-back $K_*$ of $I$ which is contained in $I$ such that
the following holds:
\begin{equation}\label{NS}
n-n_1\leq r(K_*)\leq(1+\varepsilon)n;\end{equation}
\begin{equation}\label{NS''} | K_*|\geq\frac{K_\tau C(\varepsilon) }{n}|K|;
\end{equation}
\begin{equation}\label{NS'}\frac{1}{r(K_*)}S_{r(K_*)}(x)\geq \alpha_i-\varepsilon\quad \text{for every }x\in K_*\text{ and every } i\in\{1,\ldots,l\}.\end{equation}

    For each $K\in\mathscr D_n$ define $m(K)=\max\{k\leq n\colon f^k(K)\subset I\}$
and define
$$\mathscr D_n^{'}=\{K\in\mathscr D_{n}\colon m(K)<n-n_0\}$$
and
$$\mathscr D_n^{''}=\{K\in\mathscr D_{n}\colon n-n_0\leq m(K)\leq n\}.$$
If $K\in \mathscr D_n^{''}$, then define $K_*$ to be the pull-back of $I$ by $f^{m(K)}$ which contains $K$.
Then $r(K_*)=m(K)$ and \eqref{NS} holds. \eqref{NS''} is obvious because $K=K_*$.

 Let $K\in\mathscr{D}_n^{'}$. Let $A$ denote the connected component
   of  $\{R>n-m(K)\}$ which contains $f^{m(K)}(K)$. By Proposition \ref{nse}
   and $n-m(K)\leq n$
   there exists $J\in\mathcal W$ which is contained in $A$ and satisfies
    $$n-m(K)< R(J)\leq(1+\varepsilon)(n-m(K))\quad\text{and}\quad
 |J|\geq\frac{C(\varepsilon)}{n}|A|.$$ 
 Let $J'$ denote the pull-back of $I$ by $f^{m(K)}$ which is contained in $I$
 and contains $K$.
 Let $K_*$ denote the pull-back of $J$ by $f^{m(K)}$ 
 which is contained in $J'$.  
 Since $r(K_*)=m(K)+R(J)$, 
 $$n\leq r(K_*)\leq m(K)+(1+\varepsilon)(n-m(K))\leq(1+\varepsilon)n$$
 and
\begin{equation}\label{EQ1}
\frac{|K_*|}{|K|}\geq K_\tau\frac{|f^{m(K)}(K_*)|}{|f^{m(K)}(K)|}\geq K_\tau\frac{|J|}{|A|}\geq\frac{K_\tau C(\varepsilon)}{n}.\end{equation}

 It remains to show \eqref{NS'}.
Fix $\delta>0$ such that
if $|x-y|\leq\delta$ then
$|\phi_i(x)-\phi_i(y)|\leq\varepsilon/2$ holds for every $i\in\{1,\ldots,l\}$.
For this $\delta$ let $n_2\geq1$ be the integer for which the conclusion of Lemma \ref{back} holds.

 For each $K\in\mathscr{D}_{n}$ choose $x_*\in K$ such that
$(1/n)S_n\phi_i(x_*)\geq\alpha_i$ holds for every $i\in\{1,\ldots,l\}$.
Using \eqref{NS},
for every $x\in K_*$ and every $i\in\{1,\ldots,l\}$ we have
\begin{align*}S_{r(K_*)}\phi_i(x) - S_{r(K_*)}\phi_i(x_*) &\leq| S_{r(K_*)-n_2-1}\phi_i(x) - S_{r(K_*)-n_2-1}\phi_i(x_*)|\\
&+| S_{n_2+1}\phi_i(f^{r(K_*)-n_2-1}(x)) - S_{n_2+1}\phi_i(f^{r(K_*)-n_2-1}(x_*))|\\
&\leq (r(K_*)-n_2-1)\frac{\varepsilon}{2}+2(n_2+1)\|\phi_i\|,\end{align*}
where $\|\phi_i\|=\max_{x\in X}|\phi_i(x)|$.
For sufficiently large $n$, $r(K_*)$ becomes large and
we have  $$\frac{1}{r(K_*)}S_{r(K_*)}\phi_i(x) > \alpha_i - \varepsilon,$$
which implies \eqref{NS'}.

 We are in position to finish the proof of Proposition \ref{upper0}. \eqref{NS''} gives
\begin{equation*}
\sum_{K\in\mathscr{D}_n}|K|\leq \frac{n}{K_\tau C(\varepsilon)}\sum_{K\in\mathscr{D}_n}|K_*|.
\end{equation*}
Split the summand of the right-hand-side as follows:
$$\sum_{K\in\mathscr{D}_n}|K_*|=\sum_{s=n-n_0}^{[(1+\varepsilon)n]}\sum_{\stackrel{K\in\mathscr{D}_n}
{r(K_*)=s}}|K_*|.$$
Suppose that $s_0\in\{n-n_0,\ldots,[(1+\varepsilon)n]\}$ maximizes the summand.
Then
$$\sum_{K\in\mathscr{D}_n}|K_*|\leq(\varepsilon n+n_0+1)\sum_{
\stackrel{K\in\mathscr{D}_n}
{r(K_*)=s_0}}|K_*|.$$
Combining three inequalities we get
\begin{equation*}\label{uno}
\frac{1}{n}\sum_{K\in\mathscr{D}_n}|K|\leq\frac{1}{n}\log\frac{n(\varepsilon n+n_0+1)}{K_\tau C(\varepsilon)}+\frac{1}{n}\log\sum_{
\stackrel{K\in\mathscr{D}_n}
{r(K_*)=s_0}}|K_*|.
\end{equation*}
From Lemma \ref{horse} there exists
$\mu\in\mathcal M(f)$ such that 
$\int\phi_id\mu>\alpha_i-\varepsilon$ holds for every $i\in\{1,\ldots,l\}$ and
\begin{equation*}\frac{1}{s_0}\log\sum_{\stackrel{K\in\mathscr{D}_n}
{r(K_*)=s_0}}|K_*|
\leq \mathscr{F}(\mu)+\frac{1}{s_0}\log\frac{|I|}{K_\tau}.
\end{equation*}
For sufficiently large $n$ we have
\begin{align}\label{dos}\frac{1}{n}\log\sum_{\stackrel{K\in\mathscr{D}_n}
{r(K_*)=s_0}}|K_*|
&\leq \frac{s_0}{n}\mathscr{F}(\mu)+\frac{1}{n}\log\frac{|I|}{K_\tau}\\
&=\mathscr{F}(\mu)-\frac{n_0}{n}\mathscr{F}(\mu)+\frac{1}{n}\log\frac{|I|}{K_\tau}.
\end{align}
\eqref{uno} and \eqref{dos} together imply the desired inequality for sufficiently large $n$.\end{proof}

\subsection{Overall estimate}\label{exact}

The upper bound follows from the next proposition.

\begin{prop}\label{upper}Let $f$ be a topologically exact $S$-unimodal map with non-recurrent flat critical point
that is of polynomial order. 
For every $\varepsilon>0$, every integer $l\geq1$,
let $\phi_1,\ldots,\phi_l\colon X\to\mathbb R$ be
continuous functions, and let $\alpha_1,\ldots,\alpha_l\in\mathbb R$.
Then
\begin{multline*}
\limsup_{n\to\infty}\frac{1}{n} \log \left|\left\{x\in X\colon \frac{1}{n}S_n\phi_i(x)\geq\alpha_i\ \text{ for every $i\in\{1,\ldots,l\}$}\right\}\right|\\
\\ \leq
\sup\left\{\mathscr{F}(\mu)\colon\text{$\mu\in\mathcal M(f)$ and $\int\phi_id\mu>\alpha_i-\varepsilon$ for every $i\in\{1,\ldots,l\}$}\right\}+\varepsilon.
\end{multline*}
\end{prop}

\begin{proof}
Since $f$ is topologically exact, there exists an integer $M\geq1$ such that $f^M(I)=X$. 
Let $\varepsilon>0$, $l\geq1$ an integer, 
$\phi_1,\ldots,\phi_l\colon X\to\mathbb R$ continuous and $\alpha_1,\ldots,\alpha_l\in\mathbb R$.
Since each $\phi_i$ is bounded, for sufficiently large $n$ we have
\begin{multline*}
\left\{x\in X\colon
  \frac{1}{n}S_n\phi_i(x)\geq\alpha_i\ \text{ 
     for every $i\in\{1,\ldots,l\}$}\right\}
 \subset \\
f^M\left\{x\in I\colon
  \frac{1}{n}S_n\phi_i(x)\geq\alpha_i - \varepsilon \ \text{ 
     for every $i\in\{1,\ldots,l\}$}\right\},
\end{multline*}
and therefore
\begin{multline*}
\frac{1}{n}\log\left|\left\{x\in X\colon
    \frac{1}{n}S_n\phi_i(x)\geq\alpha_i\ \text{ 
       for every $i\in\{1,\ldots,l\}$}\right\}\right|
\\ \leq
\frac{1}{n} \log\left|\left\{x\in I\colon
    \frac{1}{n}S_n\phi_j(x)\geq\alpha_i - \varepsilon \ \text{ 
       for every $i\in\{1,\ldots,l\}$}\right\}\right|+\frac{\varepsilon}{2}.
\end{multline*}
By Proposition~\ref{upper0}, for each large
$n$ there exists $\mu\in\mathcal M(f)$ such that
$\int \phi_id\mu>\alpha_i - \varepsilon$ for every $i\in\{1,\ldots,l\}$ and
$$ \frac{1}{n} \log\left|\left\{x\in I\colon
    \frac{1}{n}S_n\phi_i(x)\geq\alpha_i-\varepsilon\
    \text{ for every $i\in\{1,\ldots,l\}$}\right\}\right|
\leq
\mathscr{F}(\mu) + \frac{\varepsilon}{2}.$$
Combining the above two inequalities and then letting $n\to\infty$ yields the desired inequality.
\end{proof}

\subsection{End of the proof of the upper bound}\label{end}

\begin{proof}[Proof of the upper bound \eqref{up}]
Following \cite{ChuRivTak} let $\mathcal K$ be a closed subset of $\mathcal M$, and 
$\mathcal G$ an arbitrary open set containing~$\mathcal K$.
Since~$\mathcal K$ is compact, 
one can choose a finite collection $\mathcal C_1, \ldots, 
\mathcal C_r$ of closed sets
such that $\mathcal K \subset \bigcup_{k=1}^r \mathcal C_k \subset \mathcal G $
and each has the form
$$\mathcal C_k=\left\{\mu\in\mathcal M\colon \int\phi_j d\mu\geq\alpha_j\text{ for every $j\in\{1,\ldots,p\}$}\right\},$$
where $p\geq1$ is an integer, $\phi_j\colon X\to\mathbb R$ is continuous and 
$\alpha_j\in\mathbb R$.
For each $k\in\{1,\ldots,r\}$ and $\varepsilon>0$ define an open set
$\mathcal C_k(\varepsilon )$ containing $\mathcal C_k $
by replacing $\int\phi_j d\nu\geq \alpha_j$ in the definition of $\mathcal C_k$ by $\int\phi_j d\nu> \alpha_j-\varepsilon$.
Proposition \ref{upper} gives
$$\limsup_{n\to\infty}\frac{1}{n}\log\left|\left\{x\in X\colon
\delta_{x}^n\in\mathcal C_k\right\}\right|\leq \sup_{ \mu\in\mathcal C_k(\varepsilon)}\mathscr{F}(\mu)+\varepsilon.$$
Since
$\bigcup_{k=1}^r \mathcal C_k (\varepsilon )\subset \mathcal G$
for $\varepsilon>0$ small enough,
we have
\begin{align*}
 \limsup_{n\to\infty}\frac{1}{n}
\log\left|\left\{x\in X\colon\delta_x^n \in\bigcup_{k=1}^r {\mathcal C}_k \right\}\right|
&\leq
\max_{k\in\{1,\ldots,r\}} \limsup_{n\to\infty}\frac{1}{n}\log
\left|\left\{x\in X\colon\delta_x^n \in {\mathcal C}_k \right\}\right|\\
& \le \max_{k\in\{1,\ldots,r\}}  
\sup_{\mu\in \mathcal C_k (\varepsilon)} \mathscr{F}(\mu)  +\varepsilon
\\ & \le
\sup_{\mu\in\mathcal G}\mathscr{F}(\mu )+\varepsilon.
\end{align*}
Letting~$\varepsilon \to 0$ we obtain
$$\limsup_{n\to\infty}\frac{1}{n}\log\left|\left\{x\in X\colon \delta_x^n\in\mathcal K\right\}\right|
\le \sup_{\mu\in\mathcal G}\mathscr{F}(\mu).$$
Since $\mathcal G$ is an arbitrary open set containing 
 $\mathcal K$, it follows that
\begin{align*}
\limsup_{n\to\infty}\frac{1}{n}\log\left|\left\{x\in X\colon  \delta_{x}^n\in\mathcal K\right\}\right|
&\leq
\inf_{\mathcal G \supset \mathcal{K}} \sup_{ \mu\in \mathcal G } \mathscr{F}(\mu)\\
&=
\inf_{\mathcal G \supset \mathcal{K}} \sup_{ \mu\in \mathcal G } (-\mathscr{I}(\mu))\\
&=
-\inf_{\mu\in\mathcal K}\mathscr{I}(\mu).\end{align*}
The last equality is due to the upper semi-continuity of $-\mathscr{I}$. This completes
the proof of \eqref{up} and hence that of Theorem A. \end{proof}

\section{Structure of zeros of rate function}

In this section we analyze the structures of the set of zeros of the rate functions for maps in Theorem A. 
In Sect.\ref{analytic} we develop analytic estimates associated with the inducing scheme,
and finish the proof of Theorem B in Sect.\ref{endb}. 
The rest of this section is devoted to the proof of Theorem C.

\subsection{Recovering expansion}\label{analytic}
For the proof of Theorem B
we need the next analytic estimates
associated with the inducing scheme.
 For two positive functions $a(x)$, $b(x)$ defined on (subsets of) neighborhoods of the critical point $c$, the expression $a(x)\sim b(x)$
indicates that $a(x)/b(x)$ is bounded and bounded away from $0$.

\begin{lemma}\label{P}
Let $f$ be a topologically transitive $S$-unimodal map with a non-recurrent flat critical point $c$,
and let  $(I,\mathcal W,R)$ be an inducing scheme.
If $x\in I$ and there exists an integer $k\geq0$ such that $x\in J_k^+\cup J_k^-$
and $J_k^+\in\mathcal W$,
then
$$  R(x)\sim\ell(x)\log|x-c|^{-1},$$
and
$$ |Df^{R(x)}(x)|\sim\left|D\ell(x)\log|x-c|+\frac{\ell(x)}{x-c}\right|.$$
\end{lemma}
\begin{proof}
Let $x\in I$.
From Ma\~n\'e's hyperbolicity theorem  \cite[Theorem A]{Man85}, the distortion of iterates of $f$ outside of $I$ is uniformly bounded:
 there exists a constant $C=C(I)\geq1$
 such that for every $z\in X$ in between $f(x)$ and $f(c)$,
 \begin{equation}\label{subeq-1}C^{-1}\leq \frac{|Df^{R(x)-1}(z)|}{|Df^{R(x)-1}(f(c))|}\leq C.\end{equation}
Since $c$ is flat, up to $C^1$ changes of coordinates around $c$ and $f(c)$ we have
$f(x)=f(c)-|x-c|^{\ell(x)}$.
From the assumption on $x$, the orbit $f(x),\ldots,f^{R(x)-1}(x)$
follows the critical orbit $f(c),\ldots,f^{R(x)-1}(c)$ and as a result
for every $i\in\{1,\ldots,R(x)-1\}$ the segment connecting $f^i(x)$ and $f^i(c)$
does not intersect $I$.
Using \eqref{subeq-1} we obtain
\begin{equation}\label{criexp}
|f^{R(x)}(x)-f^{R(x)}(c)|\sim|x-c|^{\ell(x)}|Df^{R(x)-1}(f(c))|.\end{equation}
There exist constants $C>0$ and $\lambda_0>0$ such that
\begin{equation}\label{critexp}
Ce^{\lambda_0(R(x)-1)}\leq |Df^{R(x)-1}(f(c))|\leq \|Df\|^{R(x)-1},\end{equation}
where $\|Df\|=\max\{Df(x)\colon x\in X\}$.
Therefore there exists a constant $\tilde C=\tilde C(I)>1$ such that
\begin{equation}\label{subeq0}
\tilde C^{-1}|x-c|^{\ell(x)}e^{\lambda_0(R(x)-1)}\leq |f^{R(x)}(x)-f^{R(x)}(c)|\leq \tilde C|x-c|^{\ell(x)}\|Df\|^{R(x)-1}.\end{equation}
Since $f^{R(x)}(x)\in I$ and $f^{R(x)}(c)$ does not belong to the concentric closed interval with $I$ of length $(1+2\tau)|I|$,
$\tau|I|\leq |f^{R(x)}(x)-f^{R(x)}(c)|$ holds (See the line after \eqref{R1} for the choice of $\tau$). Plugging this into the second inequality in 
\eqref{subeq0} gives a lower estimate of $R(x)$.
Plugging $|f^{R(x)}(x)-f^{R(x)}(c)|\leq1$ into the first inequality in 
\eqref{subeq0} gives an upper estimate of $R(x)$.
These two estimates together imply the desired one.

For the derivative estimate, note that \begin{equation}\label{constraint}
\begin{split}
|Df(x)|&\sim |x-c|^{\ell(x)}\left|D\ell(x)\log|x-c|+\frac{\ell(x)}{x-c}\right|\\
&\sim |f(x)-f(c)|\cdot\left|D\ell(x)\log|x-c|+\frac{\ell(x)}{x-c}\right|.
\end{split}\end{equation}
The assumption on $\ell$ implies that the two terms in the second factor have the same sign:
positive for $x>c$ and negative for $x<c$.
Hence
\begin{align*}
|Df^{R(x)}(x)|\sim  |Df^{R(x)-1}(f(x))| \cdot    |f(x)-f(c)|\cdot\left|D\ell(x)\log|x-c|+\frac{\ell(x)}{x-c}\right|.
\end{align*}
For some $z$ in between $f(x)$ and $f(c)$, $$|f^{R(x)}(x)-f^{R(x)}(c)|=|Df^{R(x)-1}(z)|\cdot|f(x)-f(c)|$$ holds.
From this and \eqref{subeq-1} we obtain
$$|Df^{R(x)-1}(f(x))|\cdot|f(x)-f(c)|\sim |(f^{R(x)})(x)-f^{R(x)}(c)|.$$
This completes the proof of Lemma \ref{P}.
\end{proof}

\begin{remark}
Since $Df(x)\to0$ as $x\to c$, \eqref{constraint} imposes a constraint on $\ell$: $|x-c|^{\ell(x)}D\ell(x)\log|x-c|\to0$ as $x\to c$.
This condition is satisfied for a flat critical point of polynomial order.
\end{remark}

\subsection{On the proof of Theorem B}\label{endb}

\begin{proof}[Proof of Theorem B]
Let $f$ be a topologically exact $S$-unimodal map with a non-recurrent flat critical point $c$
that is of polynomial order.
Let  $\mu\in\mathcal M(f)$ be a post-critical measure.
To show $\mathscr{I}(\mu)=0$, it suffices to show that $\mu$ is weak* approximated by measures which are 
supported on periodic orbits and with arbitrarily small Lyapunov exponents. Namely,
we construct a sequence $\{n_i\}_{i\geq0}$ of positive integers 
and a sequence $\{x_i\}_{i\geq0}$ in $X$
such that $n_i\to\infty$ as $i\to\infty$, $f^{n_i}(x_i)=x_i$ for each $i$ and the following holds:
\begin{itemize}
 \item[(i)]  for each continuous $\phi \colon X\to\mathbb R$,
 $\displaystyle \left| \int \phi  d\delta_{x_i}^{n_i}  -  \int \phi  d\mu  \right| \to 0$
 as $i\to\infty$;
 \item[(ii)] $\chi(\delta_{x_i}^{n_i}  )\to0$ as $i\to\infty$.
  \end{itemize}

Let $(I,\mathcal W,R)$ be the inducing scheme constructed in the proof of Proposition \ref{nse}.
Let $k\geq0$ be an integer.
If $J_k^+\in\mathcal W$ then the closure of $J_k^+$ contains a periodic point of period $R_{k}$,
denoted by $y_k$.
 Since $c$ is of polynomial order, we have
\begin{equation}\label{lyapunovlimit}\displaystyle{\lim_{x\to c+0}}\frac{\log|D\ell(x)|}{\ell(x)\log|x-c|^{-1}}=0.\end{equation}
 Since $y_k\to c$,
the estimates in Lemma \ref{P} and \eqref{lyapunovlimit} together imply
\begin{equation*}\lim_{k\to\infty}\chi(\delta_{y_k}^{R_k}  )=\lim_{k\to\infty}\frac{1}{R_k}\log|Df^{R_k}(y_k)|=0.\end{equation*}
Moreover, $f^{R_k-1}$ maps $f(I_k)$ diffeomorphically onto its image.
From Ma\~n\'e's hyperbolicity theorem  \cite[Theorem A]{Man85}, for every $\delta>0$ there exists an integer $N(\delta)\geq1$ such that
if $R_k>N(\delta)$ then
$|f^n(I_k)|\leq\delta$ holds for every $n\in\{1,\ldots,R_k-N(\delta)\}$.

Since $\mu$ is a post-critical measure,
there exists a sequence $\{m_i\}_{i\geq0}$ of positive integers
such that $m_i\nearrow\infty$ and $\delta_{c}^{m_i}\to\mu$ weakly as $i\to\infty$.
For each $i$ let $\xi(i)\geq0$ be the integer with
$R_{\xi(i)}\leq m_i <R_{\xi(i)+1}$.
If $J_{\xi(i)}^+\in\mathcal W$, then put $n_i=R_{\xi(i)}$ and
$x_i=y_{\xi(i)}$.
If $J_{\xi(i)}^+\notin\mathcal W$, then $J_{\xi(i)+1}^+\in\mathcal W$ holds.
Put $n_i=R_{\xi(i)+1}$ and $x_i=y_{\xi(i)+1}$.
From the construction, 
 $n_i\to\infty$, $f^{n_i}(x_i)=x_i$ and  $\chi(\delta_{x_i}^{n_i}  )\to0$ as $i\to\infty$.

It remains to show (i).
Let $\phi\colon X\to\mathbb R$ be continuous.
For every $\varepsilon>0$ there exists $\delta>0$ such that if $x,y\in X$ and $|x-y|\leq\delta$
then $|\phi(x)-\phi(y)|\leq\varepsilon$.
If $n_i>N(\delta)$, then
$ \left| S_{n_i} \phi(x_i)  -  S_{n_i} \phi(c) \right| \leq n_i\varepsilon+N(\delta)\|\phi\|,$
 where $\|\phi\|=\max_{x\in X}|\phi(x)|$.
Since $|m_i-n_i|\leq R(V_0^+)$ from \eqref{R1},
$ \left| (1/n_i)S_{n_i} \phi(x_i)  -  (1/m_i)S_{m_i} \phi(c) \right| \leq2\varepsilon$
holds for sufficiently large $i$.
In other words, $ \left| (1/n_i)S_{n_i} \phi(x_i)  -  (1/m_i)S_{m_i} \phi(c) \right| \to 0$ as $i\to\infty$.
Since
$ \left| (1/m_i)S_{m_i} \phi(c) -\int \phi d\mu\right| \to 0$ 
it follows that
$ \left| (1/n_i)S_{n_i} \phi(x_i)  - \int \phi d\mu \right| \to 0$.
This completes the proof of Theorem B.
\end{proof}

The next corollary 
is of independent interest.
\begin{cor}\label{lack}
Let $f$ be a topologically transitive $S$-unimodal map with a flat critical point $c$ such that
$$ \displaystyle{\lim_{x\to c}}\frac{\log|D\ell(x)|}{\ell(x)\log|x-c|^{-1}}=0.$$
Then $$\inf\{\chi(\mu)\colon\mu\in\mathcal M(f)\}=0.$$
In addition, there is no measure which minimizes the Lyapunov exponent.
\end{cor} 
\begin{proof}
From the proof of Theorem B, one can create a periodic measure whose Lyapunov exponent is arbitrarily small.
The last assertion follows from the first one and Lemma \ref{positive}.
\end{proof}

\subsection{Outline of the proof of Theorem C}
Theorem C immediately follows from the next proposition.

\begin{prop}\label{c}
Let $f$ be a topologically exact $S$-unimodal map with a non-recurrent flat critical point $c$
that is of polynomial order. Then the following holds:
\begin{itemize}
\item[(i)] Assume $f$ has an acip. If
$\mu\in\mathcal M(f)$ and $\mathscr{I}(\mu)=0$, then either $\mu(\omega(c))=1$ or 
$\mu_{\rm ac}$ is absolutely continuous with respect to $\mu$;

\item[(ii)] Assume $f$ has an acip and the topological entropy of $\omega(c)$ is zero.
 If
$\mu\in\mathcal M(f)$ and $\mathscr{I}(\mu)=0$, then there exist $p\in[0,1]$ and $\nu\in\mathcal M(f)$
such that $\nu(\omega(c))=1$ and $\mu=p\nu+(1-p)\mu_{\rm ac}$;

\item[(iii)] Assume $f$ has no acip. If
$\mu\in\mathcal M(f)$ and $\mathscr{I}(\mu)=0$, then $\mu(\omega(c))=1$.
\end{itemize}
\end{prop}

\begin{proof}[Proof of Theorem C]
Assume $f$ has an acip $\mu_{\rm ac}$. Then $\mathscr{I}(\mu_{\rm ac})=0$ holds.
 Since $\delta(c)$ is a post-critical measure, Theorem B gives $\mathscr{I}(\delta(c))=0$.
 Since entropy and Lyapunov exponents are affine, the set of zeros of $\mathscr{I}$ is a convex set,
 and therefore $\{p\delta(c)+(1-p)\mu_{\rm ac}\colon 0\leq p\leq1\}\subset\{\mu\in\mathcal M(f)\colon \mathscr{I}(\mu)=0\}$. From Proposition \ref{c}(ii), this inclusion is an equality. 
\end{proof}

By the definition of the rate function \eqref{rate}, if $\mathscr{I}(\mu)=0$ then one can take a sequence $\{\nu_k\}_{k\geq0}$
such that $\nu_k\to\mu$ weakly and $\mathscr{F}(\nu_k)\to0$ as $k\to\infty$.
For a proof of Proposition \ref{c} we need to analyze the limit behaviors of entropy and Lyapunov exponent
along a sequence of measures for a fixed map. It is well-known that the entropy is upper semi-continuous, while the Lyapunov exponent is not
lower semi-continuous.
A key ingredient to overcome the lack of lower semi-continuity of Lyapunov exponent is 
Lemma \ref{sentak1} which allows us to bound from below the amount of drop of Lyapunov exponent in the limit.
We prove Proposition \ref{c} by combining these ingredients with the result of Dobbs and Todd \cite{DobTod} on the upper semi-continuity of
free energies applied to a fixed map.

\subsection{Continuity of Lyapunov exponent}\label{continue}
In this and the next subsections we prove key ingredients needed for the proof of Proposition \ref{c}.
The next lemma implies that the Lyapunov exponent is continuous
on the set $\{\mu\in\mathcal M(f)\colon\mu(\omega(c))=0\}$.
\begin{lemma}\label{usc}
Let $f$ be an $S$-unimodal map with a non-recurrent critical point $c$.
Let $\{\mu_n\}_{n\geq0}$ be a sequence of ergodic measures in $\mathcal M(f)$ such that
 $\mu_n\to\mu$ weakly as $n\to\infty$ and $\mu(\omega(c))=0$.
Then $\chi(\mu_n)\to\chi(\mu)$ as $n\to\infty$.
\end{lemma}
\begin{proof}
Let $(I,\mathcal W,R)$ be an inducing scheme.
For each $\varepsilon>0$ with $\inf\{|f^n(c)-x|\colon x\in I\setminus\{c\},\ n\geq1\}>\varepsilon$ put
$$M_\varepsilon=\sup\{R(x)-\min\{n\geq1\colon |f^n(x)-f^n(c)|>\varepsilon\}\colon x\in I\setminus\{c\}\}.$$
From the non-recurrence of $c$ and Ma\~n\'e's hyperbolicity theorem \cite[Theorem A]{Man85},
$M_\varepsilon<\infty$ holds.
Fix $\varrho=\varrho(\varepsilon)>0$ such that 
$$f(B_{\varrho}(c))\subset B_{\varepsilon}(f(c))\ \text{ and }\
\inf_{x\in B_{\varrho}(c)\setminus \{c\}}|Df^{R(x)}(x)|\geq \|Df\|^{M_\varepsilon}.$$
Note that $M_\varepsilon\to\infty$, and so $\varrho\to0$ as $\varepsilon\to0$.
Put 
$$U_{\varepsilon}=B_{\varrho}(c)\cup \bigcup_{n\geq1}B_{\varepsilon}(f^n(c)).$$
The $U_\varepsilon$ decreases as $\epsilon\to0$ with
$\bigcap_{\varepsilon>0}U_\varepsilon=\overline{\{f^n(c)\colon n\geq0\}}.$ 
Fix a partition of unity $\{\rho_{0,\varepsilon},\rho_{1,\varepsilon}\}$ on $X$ such that
${\rm supp}(\rho_{0,\varepsilon})=\overline{\{x\in X\colon\rho_{0,\varepsilon}(x)>0\}}\subset U_{2\varepsilon}$
and
${\rm supp}(\rho_{1,\varepsilon})\subset X\setminus U_{\varepsilon}$.
We have
$\rho_{0,\varepsilon}\equiv1$ on $U_\varepsilon$.


Let $\{\mu_n\}_{n\geq0}$ be a sequence of ergodic measures in the statement of Lemma \ref{usc}.
Taking a subsequence if necessary we may assume the limit
$\chi_0=\displaystyle{\lim_{n\to\infty}\chi(\mu_n)}$ exists.

Since $\rho_{1,\varepsilon}\log|Df|$ is continuous and $\mu_n\to\mu$ weakly, 
$$\lim_{n\to\infty}\int\rho_{1,\varepsilon}\log|Df|d\mu_n=\int\rho_{1,\varepsilon}\log|Df|d\mu.$$
Since $\chi(\mu_n)=\int\rho_{0,\varepsilon}\log|Df|d\mu_n+\int\rho_{1,\varepsilon}\log|Df|d\mu_n$,
letting $n\to\infty$ we have
$$\chi_0=\lim_{n\to\infty}\int\rho_{0,\varepsilon}\log|Df|d\mu_n+
\int\rho_{1,\varepsilon}\log|Df|d\mu.$$
In the next two paragraphs below we show
\begin{equation}\label{below}\lim_{n\to\infty}\int\rho_{0,\varepsilon}\log|Df|d\mu_n\leq \int\rho_{0,\varepsilon}\log|Df|d\mu\end{equation}
and
\begin{equation}\label{below'}\lim_{n\to\infty}\int\rho_{0,\varepsilon}\log|Df|d\mu_n\geq0.\end{equation}
Since $\mu(\omega(c))=0$ and $\mu(\overline{\{f^n(c)\colon n\geq0\}}\setminus\omega(c))=0$, $\rho_{0,\varepsilon}\log|Df|\to0$ as $\varepsilon\to0$  $\mu$-a.e. and thus by \eqref{below} \eqref{below'} and Fatou's lemma,
$$0\leq\liminf_{\varepsilon\to0}\lim_{n\to\infty}\int\rho_{0,\varepsilon}\log|Df|d\mu_n\leq \limsup_{\varepsilon\to0}\int\rho_{0,\varepsilon}\log|Df|d\mu\leq0.$$
Also, $\rho_{1,\varepsilon}\log|Df|\to\log|Df|$ as $\varepsilon\to0$ $\mu$-a.e. and 
$$\lim_{\varepsilon\to0}\int\rho_{1,\varepsilon}\log|Df|d\mu=\chi(\mu).$$
Hence we obtain $\chi_0=\chi(\mu)$.

To show \eqref{below}, for each $m\geq1$ define
$g_m=\max\{\rho_{0,\varepsilon}\log|Df|,-m\}.$
Then $g_m$ is continuous,
$g_m\geq\rho_{0,\varepsilon}\log|Df|$ and
$g_m\to\rho_{0,\varepsilon}\log|Df|$ as $m\to\infty$ $\mu$-a.e.
 Since $\rho_{0,\varepsilon}=1$ near $c$,
$\rho_{0,\varepsilon}\log|Df|$ is $\mu$-integrable.
From the Dominated Convergence Theorem, for every $\varepsilon>0$ there exists an integer $m$ such that
$\int g_md\mu\leq\int\rho_{0,\varepsilon}\log|Df|d\mu+\varepsilon$.
Since $\mu_n\to\mu$, for sufficiently large $n$ we have
\begin{align*}
\int \rho_{0,\varepsilon}\log|Df|d\mu_n-\varepsilon&\leq\int g_md\mu_n-\varepsilon\\
&\leq\int g_md\mu\\
&\leq\int\rho_{0,\varepsilon}\log|Df|d\mu+\varepsilon,\end{align*}
and therefore $$\limsup_{n\to\infty}\int\rho_{0,\varepsilon}\log|Df|d\mu_n\leq
\int\rho_{0,\varepsilon}\log|Df|d\mu+2\varepsilon.$$ Since $\varepsilon>0$ is arbitrary, \eqref{below} holds.

It is left to show \eqref{below'}.
Since $\mu_n$ is ergodic, it is possible to choose a point $x_n\in X$ such that $f^m(x_n)\neq c$ for every $m\geq0$ and
$$\lim_{m\to\infty}\frac{1}{m}\sum_{i=0}^{m-1}\rho_{0,\varepsilon}(f^i(x_n))\log|Df(f^i(x_n))|=\int\rho_{0,\varepsilon}\log|Df|d\mu_n.$$
If $f^m(x_n)\in B_{\varrho}(c)$ for only finitely many $m\geq0$, then Ma\~n\'e's hyperbolicity theorem  \cite[Theorem A]{Man85} implies
$\int\rho_{0,\varepsilon}\log|Df|d\mu_n\geq0$. If $f^m(x_n)\in B_{\varrho}(c)$ for infinitely many $m\geq0$,
then the orbit of $x_n$ is a concatenation of segments of the form $y\in B_{\varrho}(c),
f(y),\ldots,f^{R(y)-1}(y)$. For each such a segment,
$$\sum_{i=0}^{R(y)-1}\rho_{0,\varepsilon}(f^i(y))\log|Df(f^i(y))|\geq \log|Df^{R(y)}(y)|-M_\varepsilon\log\|Df\|\geq0.$$
This implies \eqref{below'}.
\end{proof}

\begin{remark}
 The assumption $\mu(\omega(c))=0$ in Lemma \ref{usc} is not removable.
 Indeed, for maps as in Lemma \ref{usc} it is possible to show that  the Lyapunov exponent is not lower semi-continuous
 at each post-critical measure.
\end{remark}

 


\subsection{Limit behavior of Lyapunov exponents}\label{limitb}
The next lemma gives a lower bound
on the amount of drop of Lyapunov exponents
  of measures in the weak* limit.

\begin{lemma}\label{sentak1}
Let $f$ be an $S$-unimodal map with a non-recurrent flat critical point $c$.
Let $\{\mu_k\}_k$ be a sequence of ergodic measures
 in $\mathcal M(f)$ such that
 $\mu_k\to\mu\in\mathcal M(f)$ 
weakly as $k\to\infty $, where
$\mu=p\nu+(1-p)\nu_\bot$, $\nu,\nu_\bot\in\mathcal M(f)$,
$\nu(\omega(c))=1$, $\nu_{\bot}(\omega(c))=0$ and $0\leq p\leq 1$.
Then $$\liminf_{k\to\infty}\chi(\mu_k)\geq(1-p)\chi(\nu_{\bot}).$$
\end{lemma}

\begin{proof}
If there exist infinitely many $k$ such that 
the support of $\mu_k$ is contained in $\omega(c)$, then $p=1$
and the inequality holds.
In what follows we assume the number of such $k$ is finite.

For $x\in X$ and $r>0$ define $B_r(x)=[x-r,x+r]\cap X$.
For each integer $m\geq1$
fix $\alpha_m>0$ such that $B_{\alpha_m}(c)\cap\omega(c)=\emptyset$, $\alpha_m\to0$ as $m\to\infty$ and
$$\inf_{x\in B_{\alpha_m}(c)\setminus\{c\}}\left|D\ell(x)\log|x-c|+\frac{\ell(x)}{x-c}\right|\geq m^2.$$
We have used the assumptions on $\ell$.
Set $S=\{n\geq1\colon |Df(f^n(c))|<2\}.$
For each $n\in S$ define 
$k(n)=\min\{i>1\colon |Df^{i}(f^n(c))|\geq2\}.$
Set
$$V_m=B_{\alpha_m}(c)\bigcup\left(\bigcup_{n\in S}
\bigcup_{i=0}^{k(n)-1}f^{i}(B_{1/m}(f^{n}(c)))\right)\bigcup\left(\bigcup_{n\notin S} B_{1/m}(f^n(c))\right).$$
The non-recurrence of $c$ implies $\sup_{n\in S}k(n)<\infty$ and we have $$\bigcap_{m\geq1}V_m=\overline{\{f^n(c)\colon n\geq0\}}.$$ 
From the bounded distortion, the following holds for 
  sufficiently large $m$: 
  for every $n\in S$ and every
$x\in B_{1/m}(f^n(c))$, $|Df^{k(n)}(x)|\geq1$.
For every $n\geq1$ such that $n\notin S$ and every 
$x\in B_{1/m}(f^n(c))$, $|Df(x)|\geq1$.

Note that $\{V_m\}_{m\geq1}$ has the following property: 
there exists $m_0\geq1$
such that  if $m\geq m_0$, $x\in X$ and $q\geq1$ are such that 
$x,f(x),\ldots,f^{q-1}(x)\in V_m$ and $f^{q}(x)\notin V_m$, then $|Df^{q}(x)|\geq 1.$
If $x\in B_{1/m}(f^n(c))$ holds for some $n\geq1$, then this 
follows from the definition of $V_m$.
If $x\in B_{\alpha_m}(c)$, then
 since $\bigcup_{n\geq1}B_{1/m}(f^n(c))\subset V_m,$
$|f^q(x)-f^q(c)|\geq1/m$ holds.
Hence \begin{align*}
|Df^{q}(x)|&=|Df^{q-1}(f(x))|\cdot|Df(x)|\\
&\sim |Df^{q-1}(f(x))|  \cdot   |f(x)-f(c)|\cdot\left|D\ell(x)\log|x-c|+\frac{\ell(x)}{x-c}\right|\\
&\sim |f^q(x)-f^q(c)|\cdot\left|D\ell(x)\log|x-c|+\frac{\ell(x)}{x-c}\right|.
\end{align*}
This number is comparable to $m$, and therefore
$|Df^q(x)|\geq1$ provided $m$ is sufficiently large.

For each $m\geq m_0$ such that $|Df|<1$ on $B_{\alpha_m} (c)$,
define a continuous function $\varphi_m\colon X\to\mathbb R$ by
 $$\varphi_{m}(x)=\begin{cases}\max\{\log|Df(x)|,-m\}&\text{ if $x\in B_{\alpha_m}(c)$};\\
\log|Df(x)|&\text{ otherwise.}\end{cases}$$
Let $1_m$ denote the indicator function of $V_m$.
For each $\mu_k$ take a point $x_k\in X$ such that the following holds:
$$\lim_{n\to\infty}\frac{1}{n}\sum_{i=0}^{n-1}\log|Df(f^i(x_k))|= \chi(\mu_k);$$
$$\lim_{n\to\infty}\frac{1}{n}\sum_{i=0}^{n-1}1_m(f^i(x_k)\varphi_m(f^i(x_k))= \int1_m\varphi_md\mu_k;$$
$$\lim_{n\to\infty}\frac{1}{n}\sum_{i=0}^{n-1}\phi(f^i(x_k))= \int\phi d\mu_k
\enspace\text{for every continuous $\phi\colon X\to\mathbb R$}.$$
Since the support of  $\mu_k$ is not contained in $\omega(c)$ by the initial assumption,
for every sufficiently large $m$, $f^n(x_k)\notin V_m$ holds for infinitely many $n\geq0$.
Let $\{n_l\}_{l\geq1}$ denote the subsequence obtained by aligning the elements of the set
$\{n\geq0\colon f^n(x_k)\notin V_m\}$ in the increasing order.
The afore-mentioned property of $\{V_m\}_{m\geq1}$ implies
\begin{align*}\sum_{n=0}^{n_l-1}\log|Df(f^n(x_k))|&\geq
\sum_{\stackrel{0\leq n\leq n_l-1}{f^n(x_k)\notin V_m}}\log|Df(f^n(x_k))|\\
&=\sum_{\stackrel{0\leq n\leq n_l-1}{f^n(x_k)\notin V_m}}\varphi_m(f^n(x_k))\\
&=\sum_{n=0}^{n_l-1}\varphi_m(f^n(x_k))
-\sum_{n=0}^{n_l-1}1_m(f^n(x_k))\varphi_m(f^n(x_k)).
\end{align*}
On the second summand of the last line, 
\begin{align*}\lim_{l\to\infty}\frac{1}{n_l}\sum_{n=0}^{n_l-1}1_m(f^n(x_k))\varphi_m(f^n(x_k))
&=\int 1_m\varphi_md\mu_k\\
&=\int_{V_m}\varphi_md\mu_k\\
&\leq\int_{V_m\setminus B_{\alpha_m}(c)}\varphi_md\mu_k\\
&=\int_{V_m\setminus B_{\alpha_m}(c)}\log|Df|d\mu_k.
\end{align*}
The inequality holds provided $m$ is sufficiently large so that $\varphi_m$ is negative on $B_{\alpha_m} (c)$.
Hence
\begin{align*}\chi(\mu_k)
&=\lim_{l\to\infty}\frac{1}{n_l}\sum_{n=0}^{n_l-1}\log|Df(f^n(x_k))|\\
&\geq \int\varphi_md\mu_k-\int_{V_m\setminus B_{\alpha_m}(c)}\log|Df|d\mu_k.\end{align*}
Since $\mu_k\to\mu$ weakly as $k\to\infty$, 
$$\liminf_{k\to\infty}\chi(\mu_k)\geq \int\varphi_md\mu-   p\int_{V_m\setminus B_{\alpha_m}(c)}\log|Df|d\nu-(1-p)\int_{V_m\setminus 
B_{\alpha_m}(c)}\log|Df|d\nu_\bot.$$
From the Dominated Convergence Theorem, $\int\varphi_{m,1}d\mu\to\chi(\mu)$
as $m\to\infty$. Since $\omega(c)$ is contained in $V_m\setminus B_{\alpha_m}(c)$,
the second integral is equal to $\chi(\nu)$.
From $\nu_\bot(\omega(c))=0$ and  $\nu_\bot(\overline{\{f^n(c)\colon n\geq1\}}\setminus\omega(c))=0$,
the third integral goes to $0$ as $m\to\infty$.
This finishes the proof of Lemma \ref{sentak1}.
\end{proof}

\subsection{Approximation by ergodic measures}\label{approx}
We need Dobbs' extension \cite{Dob14}
of Ledrappier's characterization of acips \cite{Led81}.
 \begin{thm}{\rm (c.f. \cite[Theorem 1.5]{Dob14})}\label{dob}
 Let $f$ be an $S$-unimodal map. A measure $\mu\in\mathcal M(f)$ with $\chi(\mu)>0$ 
 is an acip if and only if $\mathscr{F}(\mu)=0$.
 \end{thm}

The next lemma asserts that the zeros of the rate function $\mathscr{I}$ are approximated by ergodic measures
with similar free energies.
This conclusion is necessary to use Lemma \ref{sentak1}.
\begin{lemma}\label{new}
Let $f$ be a topologically exact $S$-unimodal map with non-recurrent critical point.
 Let $\mu\in\mathcal M(f)$ and suppose $\mathscr{I}(\mu)=0$.
 There exists a sequence
 $\{\mu_k\}_{k\geq0}$
 in $\mathcal M(f)$ such that each $\mu_k$ is ergodic,
 $\mu_k\to\mu$ weakly
and $\mathscr{F}(\mu_k)\to0$  as $k\to\infty $.
\end{lemma}
\begin{proof}
Since $\mathscr{I}(\mu)=0$ it is possible to
take a sequence $\{\xi_k\}_{k\geq0}$ in $\mathcal M(f)$ such that $\mathscr{F}(\xi_k)\to0$ and $\xi_k\to\mu$
weakly as $k\to\infty$. 
Write  $\xi_k=p_{k}\nu_{k}+(1-p_k)\nu_{k,\bot}$ 
where 
$0\leq p_k\leq1$, $\nu_{k},\nu_{k,\bot}\in\mathcal M(f)$, $\nu_{k}(\omega(c))=1$ and
 $\nu_{k,\bot}(\omega(c))=0$. Taking a subsequence if necessary we may assume $p_k\to p$
 as $k\to\infty$.
   Ruelle's inequality \cite{Rue78} implies
 $$\limsup_{k\to\infty}\mathscr{F}(\xi_k)\leq p\limsup_{k\to\infty}\mathscr{F}(\nu_{k})\leq0.$$
 Since $\mathscr{F}(\xi_k)\to0$, the second inequality is an equality, namely
$p\displaystyle{\limsup_{k\to\infty}}\mathscr{F}(\nu_{k})=0$.
If $p\neq0$ then $\displaystyle{\limsup_{k\to\infty}}\mathscr{F}(\nu_{k})=0$.
Since $\omega(c)$ is a hyperbolic set, $\mathscr{F}$ is upper semi-continuous on the set of measures
supported on $\omega(c)$. Hence, there exists $\nu\in\mathcal M(f)$ such that
 $\nu(\omega(c))=1$ and $\mathscr{F}(\nu)=0$. 
 By Theorem \ref{dob}, $\nu$ is absolutely continuous with respect to the Lebesgue measure. 
 This is a contradiction. 
 Hence $p=0$, and thus
$|\mathscr{F}(\xi_k)-\mathscr{F}(\nu_k)|\to0$ and $\nu_k\to\mu$
as $k\to\infty$.
Since $f$ is topologically exact, it has the specification. Then, ergodic measures are entropy-dense
\cite{EKW94}: if $\mu\in\mathcal M(f)$ is non-ergodic, there exists a sequence $\{\mu_l\}_{l\geq0}$ in 
$\mathcal M(f)$ 
such that each $\mu_l$ is ergodic, $\mu_l\to\mu$ weakly and $h(\mu_l)\to h(\mu)$ as $l\to\infty$.
Hence, for each $\nu_k$ there exists a sequence $\{\nu_{k,l}\}_{l\geq0}$ of ergodic measures
such that $\nu_{k,l}\to\nu_k$ weakly and $h(\nu_{k,l})\to h(\nu_k)$ as $l\to\infty$. 
Since $\nu_k(\omega(c))=0$, Lemma \ref{usc} gives $\chi(\nu_{k,l})\to\chi(\nu_k)$ as $l\to\infty$,
and hence $\mathscr{F}(\nu_{k,l})\to \mathscr{F}(\nu_k)$. \end{proof}

\begin{remark}
Define $\mathscr{G} \colon \mathcal M \to [-\infty, 0]$ by
$$ \mathscr{G}(\nu)
=
\begin{cases}h(\nu)-\chi(\nu) &\text{ if $\nu\in\mathcal M(f)$ and is ergodic};\\
-\infty&\text{ otherwise.}\end{cases}$$
A close inspection of the proof of the upper bound \eqref{up} in Section 3 
and that of the lower bound \eqref{low} in \cite[Proposition 4.1]{ChuRivTak} shows that, for a map $f$ as in Theorem A
the rate function $\mathscr{I}$ is also given by
$ \mathscr{I}(\mu)
=
-\inf_{\mathcal G \ni \mu}\sup_{\nu\in\mathcal G}\mathscr{G}(\nu),$
where the infimum is taken over all open subsets~$\mathcal G$ of~$\mathcal M$ containing~$\mu$.
This implies the conclusion of Lemma \ref{new}.
\end{remark}

\subsection{End of the proof of Theorem C}\label{endc}

We are in position to finish the proof of Proposition \ref{c}.



\begin{proof}[Proof of Proposition \ref{c}]
Let $f$ be a topologically exact $S$-unimodal map with a non-recurrent flat critical point $c$.
Assume $f$ has an acip. 
Let $\mu\in\mathcal M(f)$ be such that $\mathscr{I}(\mu)=0$. Write
 $\mu=p\nu+(1-p)\nu_\bot$, $0\leq p\leq 1$, $\nu,\nu_\bot\in\mathcal M(f)$, 
$\nu(\omega(c))=1$, $\nu_\bot(\omega(c))=0$.
If $p=1$ then $\mu(\omega(c))=1$.
Assume $p\neq1$.
 Since $\mathscr{I}(\mu)=0$, by Lemma \ref{new}
 there is a sequence $\{\mu_k\}_{k\geq0}$ of ergodic measures in $\mathcal M(f)$ 
 such that $\mu_k$ converges weakly to $\mu$ and $\mathscr{F}(\mu_k)\to0$ as $k\to\infty$.
 By Lemma \ref{sentak1} and Lemma \ref{positive}, $\displaystyle{\liminf_{k\to\infty}}\chi(\mu_k)\geq(1-p)\chi(\nu_\bot)>0$,
 and hence $\displaystyle{\liminf_{k\to\infty}}h(\mu_k)\geq(1-p)\chi(\nu_\bot)>0$. 
From \cite[Theorem 1.18]{DobTod}, the acip of $f$ is absolutely continuous with respect to $\mu$.
 
 Assume $f$ has an acip and the topological entropy of $\omega(c)$ is zero.
Let $\mu\in\mathcal M(f)$ be such that $\mathscr{I}(\mu)=0$. Write
 $\mu=p\nu+(1-p)\nu_\bot$, $0\leq p\leq 1$, $\nu,\nu_\bot\in\mathcal M(f)$, 
$\nu(\omega(c))=1$, $\nu_\bot(\omega(c))=0$.
If $p=1$ then $\mu(\omega(c))=1$.
Assume $p\neq1$.
 Since $\mathscr{I}(\mu)=0$, by Lemma \ref{new}
 there is a sequence $\{\mu_k\}_{k\geq0}$ of ergodic measures in $\mathcal M(f)$ 
 such that $\mu_k\to\mu$ weakly and $\mathscr{F}(\mu_k)\to0$ as $k\to\infty$.
 By Lemma \ref{sentak1}, $\displaystyle{\liminf_{k\to\infty}}\chi(\mu_k)\geq(1-p)\chi(\nu_\bot)$,
Using this and the upper semi-continuity of entropy,
\begin{align*}
0&=\lim_{k\to\infty}\mathscr{F}(\mu_k)\\&\leq
\limsup_{k\to\infty}h(\mu_k)-\liminf_{k\to\infty}\chi(\mu_k)\\
&\leq h(\mu)-(1-p)\chi(\nu_\bot)\\
&=(1-p)\mathscr{F}(\nu_\bot),
\end{align*}
where the last equality is because $h(\mu)=(1-p)h(\nu_\bot)$,
from $h(\nu)=0$.
Hence $\mathscr{F}(\nu_\bot)=0$ holds. 
By Theorem \ref{dob}, $\nu$ is an acip of $f$. 
   This proves (ii).
 
 Assume $f$ has no acip.
Let $\mu\in\mathcal M(f)$ be such that $\mathscr{I}(\mu)=0$. Write
 $\mu=p\nu+(1-p)\nu_\bot$, $0\leq p\leq 1$, $\nu,\nu_\bot\in\mathcal M(f)$, 
$\nu(\omega(c))=1$, $\nu_\bot(\omega(c))=0$.
If $p=1$ then $\mu(\omega(c))=1$. Assume $p\neq 1$.
 Since $\mathscr{I}(\mu)=0$, by Lemma \ref{new}
 there is a sequence $\{\mu_k\}_{k\geq0}$ of ergodic measures in $\mathcal M(f)$ 
 such that $\mu_k$ converges weakly to $\mu$ and $\mathscr{F}(\mu_k)\to0$ as $k\to\infty$.
 By Lemma \ref{sentak1} and Lemma \ref{positive}, $\displaystyle{\liminf_{k\to\infty}}\chi(\mu_k)\geq(1-p)\chi(\nu_\bot)>0$,
 and hence $\displaystyle{\liminf_{k\to\infty}}h(\mu_k)\geq(1-p)\chi(\nu_\bot)>0$. 
From \cite[Theorem 1.18]{DobTod} there exists $\mu'\in\mathcal M(f)$
with $h(\mu')>0$ and $\mathscr{F}(\mu')=0$.
 By Theorem \ref{dob}, $\mu'$ is the acip of $f$,
   a contradiction.
 This proves (iii). The proof of Proposition \ref{c} and hence that of Theorem C is complete.
 \end{proof}

 \subsection{Convergence of the acips to the Dirac measure}\label{more}
Lastly we treat the test family $\{f_b\}_{b>0}$ given by \eqref{fb}
and prove the next proposition.
\begin{prop}\label{converge}
The acip of $f_b$, $b<1$ converges weakly to $\delta_0$ as $b\nearrow1$.
\end{prop}

A proof of Proposition \ref{converge} involves essentially the same set of ideas as that of the proof of Proposition \ref{c} (iii).
In particular, we exploit the fact that $f_{1}$ has no acip.
The difference from the proof of Proposition \ref{c} (iii) is that we need to treat a sequence of measures which are not invariant for a single fixed map.
We begin by proving a version of Lemma \ref{sentak1} which holds for such a sequence 
associated with a convergent sequence of maps in the family $\{f_b\}_{b>0}$.
This can be achieved with a minor modification, primarily because this family lies in the same topological conjugacy class. 
We finish the proof of Proposition \ref{converge} by combining this result with that of Dobbs and Todd \cite{DobTod}.

\begin{remark}
{\rm 
Let $\{b_k\}_{k\geq0}$ be a sequence in $[1/\sqrt{6},1)$ such that $b_k\nearrow1$ as $k\to\infty$, and
$\{\mu_{b_k}\}_{k\geq0}$ be a sequence of measures in $\mathcal M$
such that $\mu_{b_k}\in\mathcal M(f_{b_k})$ for each $k\geq0$, and $\mu_{b_k}$ converges 
weakly as $k\to\infty$ a measure $\mu\in\mathcal M$.
  Since $(x,b)\mapsto f_b(x)$ is continuous, 
  $\mu\in\mathcal M(f_1)$ holds.}
\end{remark}

Recall that the maps $f_b$ $(b>0)$ have $c=1/2$ as their common critical point,
and the singleton $\{0\}$ as their common omega-limit set
of the critical point.
\begin{lemma}\label{sentak}
Let $\{b_k\}_{k\geq0}$ be a sequence in $[1/\sqrt{6},1)$ such that $b_k\nearrow1$ as $k\to\infty$, and
$\{\mu_{b_k}\}_{k\geq0}$ a sequence of measures in $\mathcal M$
such that 
for each $k\geq0$, $\mu_{b_k}\in\mathcal M(f_{b_k})$, $\mu_{b_k}$ is ergodic with respect to $f_{b_k}$,
 $\mu_{b_k}$ converges to $\mu\in\mathcal M(f_1)$ 
weakly as $k\to\infty$, where
$\mu=p\delta_0+(1-p)\nu_\bot$, $\nu_\bot\in\mathcal M(f_1)$,
$\nu_{\bot}(\{0\})=0$ and $0\leq p\leq 1$.
Then $$\liminf_{k\to\infty}\chi(f_{b_k};\mu_{b_k})\geq(1-p)\chi(f_1;\nu_{\bot}).$$
\end{lemma}
\begin{proof}
We start with preliminary constructions.
Fix sequences $\{\alpha_m\}_{m\geq1}$, $\{\beta_m\}_{m\geq1}$
of positive numbers 
such that the following holds for every $m\geq1$:
$\beta_m\geq1/m$;
  $\alpha_m,\beta_m\to0$ as $m\to\infty$;
 for every $b\in [1/\sqrt{6},1]$,
  $\max_{x\in B_{\alpha_m} (c)}|Df_b(x)|<1$  and
$$\inf_{x\in B_{\alpha_m}(c)\setminus\{c\}}\left|b|x-c|^{-b-1}|\log|x-c||+\frac{|x-c|^{-b}}{|x-c|}\right|\geq m^2;$$
$\mu(\partial V_m)=0$, where $\mu$ is the measure in the statement of Lemma \ref{sentak} and
$V_m=B_{\alpha_m}(c)\bigcup B_{\beta_m}(0)\bigcup  B_{\beta_m}(1).$


We claim that $\{V_m\}_{m\geq1}$ has the following property: 
there exists an integer $m_0\geq1$ such that if $b\in[1/\sqrt{6},1]$, $m\geq m_0$,
 $x\in X$ and $q\geq1$ are such that 
$x,f_b(x),\ldots,f_b^{q-1}(x)\in V_m$ and $f_b^{q}(x)\notin V_m$, then $|Df_b^{q}(x)|\geq 1.$
Indeed, if $x\in B_{\alpha_m}(c)$ then
$|f_b^q(x)-f_b^q(c)|\geq\beta_m\geq1/m$.
Since the map $f_b$ satisfies $\ell(x)=|x-c|^{-b}$ and $|D\ell(x)|=b|x-c|^{-b-1}$ we obtain
\begin{align*}
|Df_b^{q}(x)|\sim |f_b^q(x)-f_b^q(c)|\cdot\left|b|x-c|^{-b-1}|\log|x-c||+\frac{|x-c|^{-b}}{|x-c|}\right|.
\end{align*}
This number is comparable to $m$, and therefore
$|Df_b^q(x)|\geq1$ provided $m$ is sufficiently large.
Since $|Df_b(0)|>1$ and $|Df_b(1)|>1$, if $x\in B_{\beta_m}(0)\cup B_{\beta_m}(1)$ then 
$|Df_b(x)|\geq1$ holds provided $m$ is sufficiently large. 
Hence the claim holds.

Let $\{b_k\}_{k\geq0}$ and
$\{\mu_{b_k}\}_{k\geq0}$ be the sequences in the statement of Lemma \ref{sentak}.
If $\mu_{b_k}=\delta_0$ holds for infinitely many $k$,
 then $p=1$ and the inequality holds.
 In what follows we assume $\mu_{b_k}\neq\delta_0$ for every $k\geq0$.
  For each $m\geq m_0$ and $b\in[ 1/\sqrt{6},1]$
define a continuous function $\varphi_{m,b}\colon X\to\mathbb R$ by
$$\varphi_{m,b}(x)=\begin{cases}\max\{\log|Df_b(x)|,-m\}&\text{ if $x\in B_{\alpha_m}(c)$};\\
\log|Df_b(x)|&\text{ otherwise.}\end{cases}$$
  Let $1_m$ denote the indicator function of $V_m$.
For each $\mu_{b_k}$ take a point $x_{b_k}\in X$ such that the following holds:
$$\lim_{n\to\infty}\frac{1}{n}\sum_{i=0}^{n-1}\log|Df_{b_k}(f_{b_k}^i(x_{b_k}))|= \int\log|Df_{b_k}|d\mu_{b_k};$$
$$\lim_{n\to\infty}\frac{1}{n}\sum_{i=0}^{n-1}1_m(f_{b_k}^i(x_{b_k})\varphi_{m,b_k}(f_{b_k}^i(x_{b_k}))= \int1_m\varphi_{m,b_k}d\mu_{b_k};$$
$$\lim_{n\to\infty}\frac{1}{n}\sum_{i=0}^{n-1}\phi(f_{b_k}^i(x_{b_k}))= \int\phi d\mu_{b_k}
\enspace\text{for every continuous $\phi\colon X\to\mathbb R$};$$
$$\lim_{n\to\infty}\frac{1}{n}\#\{0\leq i\leq n-1\colon f_{b_k}^i(x_{b_k})\in V_m\}=\mu_{b_k}(V_m).$$
Since $\mu(\partial V_m)=0$ and $\mu_{b_k}\to\mu$ as $k\to\infty$,
$\mu_{b_k}(V_m)\to\mu(V_m)<1$.
Hence, for every sufficiently large $m$, $f_{b_k}^n(x_{b_k})\notin V_m$ holds for infinitely many $n$.
Let $\{n_l\}_{l\geq1}$ denote the subsequence obtained by aligning the elements of the set
$\{n\geq0\colon f_{b_k}^n(x_{b_k})\notin V_m\}$ in the increasing order.
Similarly to the proof of Lemma \ref{sentak1}, the property of $\{V_m\}_{m\geq1}$ implies
\begin{align*}\sum_{n=0}^{n_l-1}\log|Df_{b_k}(f_{b_k}^n(x_{b_k}))|&\geq
\sum_{n=0}^{n_l-1}\varphi_{m,b_k}(f_{b_k}^n(x_{b_k}))
-\sum_{n=0}^{n_l-1}1_m(f_{b_k}^n(x_{b_k}))\varphi_{m,b_k}(f_{b_k}^n(x_{b_k})).
\end{align*}
On the second summand of the last line, 
\begin{align*}\lim_{l\to\infty}\frac{1}{n_l}\sum_{n=0}^{n_l-1}1_m(f_{b_k}^n(x_{b_k}))\varphi_{m,b_k}(f_{b_k}^n(x_{b_k}))
\leq\int_{V_m\setminus B_{\alpha_m}(c)}\log|Df_{b_k}|d\mu_{b_k}.
\end{align*}
Hence
\begin{equation}\label{eqlast}
\begin{split}\int\log|Df_{b_k}|d\mu_{b_k}
&=\lim_{l\to\infty}\frac{1}{n_l}\sum_{n=0}^{n_l-1}\log|Df_{b_k}(f_{b_k}^n(x_{b_k}))|\\
&\geq \int\varphi_{m,b_k}d\mu_{b_k}-\int_{V_m\setminus B_{\alpha_m}(c)}\log|Df_{b_k}|d\mu_{b_k}.
\end{split}\end{equation}

We claim that the two integrals in the right-hand side of the inequality in \eqref{eqlast} converge as $k\to\infty$.
Indeed, for each fixed $m$, $\varphi_{m,b_k}$ converges to $\varphi_{m,1}$ uniformly as $k\to\infty$.
For every $\varepsilon>0$ there exists $k_0\geq0$ such that if $k\geq k_0$ then
$\|\varphi_{m,b_k}-\varphi_{m,1}\|\leq\varepsilon/2$.
Since $\mu_{b_k}\to\mu$ weakly as $k\to\infty$ and $\varphi_{m,1}$ is continuous,
there exists $k_1\geq0$ such that if $k\geq k_1$ then
$\left|\int\varphi_{m,1}d\mu_{b_k}-\int\varphi_{m,1}d\mu\right|\leq\varepsilon/2.$
If $k\geq\max\{k_0,k_1\}$ then
\begin{align*}\left|\int\varphi_{m,b_k}d\mu_{b_k}-\int\varphi_{m,1}d\mu\right|\leq&\left|\int\varphi_{m,b_k}d\mu_{b_k}-\int\varphi_{m,1}d\mu_{b_k}\right|\\
&+
\left|\int\varphi_{m,1}d\mu_{b_k}-\int\varphi_{m,1}d\mu\right|\\
\leq&\varepsilon,\end{align*}
namely $\int\varphi_{m,b_k}d\mu_{b_k}\to\int\varphi_{m,1}d\mu$ as $k\to\infty$.
In the same way, for each fixed $m$ we have
$\int_{V_m\setminus B_{\alpha_m}(c)}\log|Df_{b_k}|d\mu_{b_k}\to  \int_{V_m\setminus B_{\alpha_m}(c)}\log|Df_{1}|d\mu$
as $k\to\infty$ and the claim holds.

Letting $k\to\infty$ in \eqref{eqlast} yields
\begin{align*}
\liminf_{k\to\infty}\int\log|Df_{b_k}|d\mu_{b_k}\geq &\int\varphi_{m,1}d\mu-   p\int_{V_m\setminus B_{\alpha_m}(c)}
\log|Df_1|d\nu\\
&-(1-p)\int_{V_m\setminus 
B_{\alpha_m}(c)}\log|Df_1|d\nu_\bot.\end{align*}
From the Dominated Convergence Theorem, $\int\varphi_md\mu\to\int\log|Df_1|d\mu$
as $m\to\infty$. Since $0\in V_m\setminus B_{\alpha_m}(c)$,
the second integral is equal to $\int\log|Df_1|d\nu$.
Since $\log|Df_1|$ is bounded on $V_m\setminus B_{\alpha_m}(c)$
and $\nu_\bot(V_m\setminus B_{\alpha_m}(c))\to0$ as $m\to\infty$,
the third integral goes to $0$ as $m\to\infty$.
This finishes the proof of Lemma \ref{sentak}.
\end{proof}

\begin{proof}[Proof of Proposition \ref{converge}]
For each $b\in[1/\sqrt{6},1)$ let $\mu_{{\rm ac},b}$ denote the acip of $f_b$.
Let $\{b_k\}_{k\geq0}$ be an arbitrary sequence in $[1/\sqrt{6},1)$ such that $b_k\nearrow1$ as $k\to\infty$ and
$\{\mu_{{\rm ac},b_k}\}_{k\geq0}$ converges weakly to a measure $\mu\in\mathcal M(f_1)$.
Write
 $\mu=p\delta_0+(1-p)\nu_\bot$, $0\leq p\leq 1$, $\nu_\bot\in\mathcal M(f_1)$, 
$\nu_\bot(\{0\})=0$.
If $p=1$ then $\mu=\delta_0$. Assume $p\neq 1$. 
 From Lemma \ref{sentak} and Lemma \ref{positive}, $\displaystyle{\liminf_{k\to\infty}}\chi(\mu_{{\rm ac},b_k})\geq(1-p)\chi(\nu_\bot)$
 holds.
 The characterization of the acip gives
 $\mathscr{F}(f_{b_k};\mu_{{\rm ac},b_k})=0$ for every $k\geq0$, and thus
  $\displaystyle{\liminf_{k\to\infty}}h(\mu_{{\rm ac},b_k})\geq(1-p)\chi(\nu_\bot)$. 
We have
 \begin{align*}
 (1-p)\chi(\nu_\bot ) &\ge  (1-p)h(\nu_\bot ) \\
 &= h(\mu) \\
 &\ge \liminf_{k\to\infty} h(\mu_{{\rm ac}, b_k}) \\
&\ge (1-p)\chi (\nu_\bot).\end{align*}
The first inequality is from Ruelle's inequality, and the second one 
from the upper semi-continuity of entropy for a sequence of maps 
\cite[Theorem 1.15]{DobTod}. It follows that all the inequalities are equalities
and $h(\nu_\bot)=\chi(\nu_\bot)$.
 By Theorem \ref{dob} the measure $\nu_\bot$ is an acip,
   a contradiction.
    \end{proof}


 \section*{Appendix A. Rate functions for Collet-Eckmann maps}\label{consequence}
In this appendix we characterize the zero of the rate function for a Collet-Eckmann map.

\begin{thmA1}\label{A1}
Let $f$ be a topologically exact $S$-unimodal map with non-flat critical point satisfying 
the Collet-Eckmann condition.
Then $\mathscr{I}(\mu)=0$ if and only if $\mu=\mu_{\rm ac}.$
\end{thmA1}

\begin{proof}
Let $\mu\in\mathcal M(f)\setminus\{\mu_{\rm ac}\}$. 
Take a Lipschitz continuous function $\phi\colon X\to\mathbb R$ such that  
 $\int \phi d\mu \not= \int \phi d\mu_{\rm ac}$. 
 Then $\phi\neq\psi\circ f-\psi$ holds for every $\psi\in L^2(\mu_{\rm ac})$,
 and thus $\sigma_\phi^2>0$, see Liverani \cite{Liv96}.
  Put  $\epsilon_0 = |\int \phi d\mu - \int \phi d\mu_{\rm ac}| /2$. 
  The set $\{\nu\in\mathcal M\colon|\int\phi d\nu-  \int\phi d\mu_{\rm ac}|>\epsilon_0\}$
  is an open subset of $\mathcal M$ which contains $\mu$.
  The lower bound in the LDP \eqref{low} gives
\begin{align*}-\mathscr{I}(\mu )& \le \liminf_{n\to\infty}\frac{1}{n}\log \left|\left\{ x\in X\colon \left |\frac{1}{n}S_n\phi(x) - \int \phi d\mu_{\rm ac}\right| > \epsilon_0   \right\}\right|.
\end{align*} 
From Theorem \ref{kelnow} the right-hand-side is strictly negative,
 and therefore $\mathscr{I}(\mu)>0$.
\end{proof}

\section*{Appendix B: LDP for intermittent maps}
In this appendix we treat the Manneville-Pomeau map $f_\alpha\colon X\to X$ 
given by $f_\alpha(x)=x+x^{1+\alpha}$ (mod $1$)
 where $f_\alpha(0)=0$, the value of $f_\alpha$ at its discontinuity is $0$, $f_\alpha(1)=1$ and $\alpha>0$.
The map $f_\alpha$ has an acip if and only if $\alpha<1$.
The acip is unique and is denoted by $\mu_{\rm ac,\alpha}$.
Let $\mathcal M(f_\alpha)$ denote the set of $f_\alpha$-invariant Borel probability measures.
For each $\mu\in\mathcal M(f_\alpha)$
 the Kolmogorov-Sina{\u\i} entropy of $(f_\alpha,\mu)$ is denoted by $h(f_\alpha,\mu)$
and $\chi(f_\alpha;\mu)=\int\log |Df_\alpha| d\mu$.
We do not mind any clash of notation with the previous sections.

\begin{thmB1}\label{thmB1}
 Let $f_\alpha$ be the Manneville-Pomeau map.
Then the Large Deviation Principle holds.
The rate function $\mathscr{I}=\mathscr{I}(f_\alpha;\, \cdot)\colon \mathcal M\to[0,\infty]$ is given by 
$$ \mathscr{I}(f_\alpha;\mu)
=
\begin{cases}\chi(\mu)-h(\mu) &\text{ if $\mu\in\mathcal M(f_\alpha)$};\\
\infty&\text{ otherwise.}\end{cases}$$
In addition, $\mu_{{\rm ac},\alpha}$ converges weakly to $\delta_0$ as $\alpha\nearrow1$.
\end{thmB1}

It follows that
there is a qualitative change in the structure of the set of zeros of 
$\mathscr{I}(f_\alpha;\mu)$ at $\alpha=1$:
for $0<\alpha<1$, $\mathscr{I}(f_\alpha;\mu)=0$ if and only if 
there exists $p\in[0,1]$ such that $\mu=p\delta_0+(1-p)\mu_{\rm ac,\alpha}$;
for $\alpha\geq1$, $\mathscr{I}(f_\alpha;\mu)=0$ if and only if $\mu=\delta_0$.
\medskip

\noindent{\it Remark B.2.}
For simplicity we have suppressed small generalizations 
to other interval maps with neutral fixed point. The statements as in Theorem B.1
hold, for example, for maps treated in Nakaishi \cite{Nak00}, Pollicott and Sharp \cite{PolSha09}.

\begin{proof}[Proof of Theorem B.1]
A proof of the lower bound \eqref{low} is almost identical to 
those of \cite[Section 7]{Chu11} and \cite[Proposition 4.1]{ChuRivTak}
and hence we omit it. The existence of the discontinuity does not matter.
We only give a proof of the upper bound \eqref{up}.

Let $I$ denote the domain of the branch of $f$ not containing $0$.
 The first return time to $I$ is a function $R\colon I\to \mathbb Z_{>0}\cup\{\infty\}$ 
defined by
$$R(x)=\inf\left(\{n\geq1\colon f_\alpha^n(x)\in I\}\cup\{\infty\}\right).$$
We show that the inducing scheme obtained from the first return map to $I$
given by 
$x\in I\mapsto f_\alpha^{R(x)}(x)$ satisfies the following specification-like property 
which is a counterpart of Proposition \ref{nse}.

\begin{lemB2}\label{lemB2}
There exists a constant $C>0$ such that
$$\frac{|\{R=n+1\}|}{|\{R>n\}|}\geq Cn^{-2(1+\alpha)/\alpha}\text{ for every $n\geq1$}.$$
\end{lemB2}
\begin{proof}
The Mean Value Theorem gives 
$ |Df_\alpha^{n+1}(x)|\cdot|\{R=n+1\}|=|f_\alpha^{n+1}(\{R=n+1\})|$ for some $x\in \{R=n+1\}$.
By \cite[Lemma 2.1]{Nak00}, there exists a constant $C>0$ such that
$|\{R=n+1\}|\geq Cn^{-2(1+\alpha)/\alpha}.$
This and $|\{R>n\}|\leq|X|=1$ together yield the desired inequality.
\end{proof}
Using Lemma B.2 in the place of Proposition \ref{nse}
and repeating the argument in Sect.3 show the upper bound \eqref{up}.
Since  $\mu\in\mathcal M(f_\alpha)\mapsto h(\mu)$ is upper semi-continuous and
 $\mu\in\mathcal M(f_\alpha)\mapsto\int\log|Df_\alpha|d\mu$ is continuous,
the rate function has the desired form.
The characterization of the zeros of the rate function follows from the result of Ledrappier \cite{Led81}.

\begin{lemB3}\label{lemB3}
 Let $\{\alpha_k\}_{k\geq0}$ be a sequence in $(0,1)$ such that $\alpha_k\nearrow1$ as $k\to\infty$, and
$\{\mu_{\alpha_k}\}_{k\geq0}$ a sequence of measures in $\mathcal M$
such that 
for each $k\geq0$, $\mu_{\alpha_k}\in\mathcal M(f_{\alpha_k})$, $\mu_{\alpha_k}$ is ergodic with respect to
 $f_{\alpha_k}$,
 $\mu_{\alpha_k}$ converges to $\mu\in\mathcal M(f_1)$ 
weakly as $k\to\infty$, where
$\mu=p\delta_0+(1-p)\nu_\bot$, $\nu_\bot\in\mathcal M(f_1)$,
$\nu_{\bot}(\{0\})=0$ and $0\leq p\leq 1$.
Then $$\lim_{k\to\infty}\chi(f_{\alpha_k};\mu_{\alpha_k})=(1-p)\chi(f_1;\nu_{\bot}).$$
\end{lemB3}
\begin{proof}
Since $\log|Df_{\alpha_k}|$ converges to $\log|Df_{1}|$ uniformly as $k\to\infty$,
for every $\varepsilon>0$ there exists $k_0\geq0$ such that if $k\geq k_0$ then
$\|\log|Df_{\alpha_k}|-\log|Df_{1}|\|\leq\varepsilon/2$.
Since $\mu_{\alpha_k}\to\mu$ weakly as $k\to\infty$ and $\log|Df_1|$ is continuous,
there exists $k_1\geq0$ such that if $k\geq k_1$ then
$\left|\int\log|Df_{1}|d\mu_{\alpha_k}-\int\log|Df_{1}|d\mu\right|\leq\varepsilon/2.$
If $k\geq\max\{k_0,k_1\}$ then

\begin{align*}\left|\int\log|Df_{\alpha_k}|d\mu_{\alpha_k}-\int\log|Df_{1}|d\mu\right|\leq&\left|\int\log|Df_{\alpha_k}|d\mu_{\alpha_k}-\int
\log|Df_1|d\mu_{\alpha_k}\right|\\
&+
\left|\int\log|Df_{1}|d\mu_{\alpha_k}-\int\log|Df_{1}|d\mu\right|\\
\leq&\varepsilon,\end{align*}
namely $\int\log|Df_{\alpha_k}|d\mu_{\alpha_k}\to\int\log|Df_1|d\mu=(1-p)\chi(f_1;\nu_{\bot})$ as $k\to\infty$.
\end{proof}

 Using Lemma B.3 in the place of Lemma \ref{sentak}
 and repeating the argument in the proof of Proposition \ref{converge} 
show that $\mu_{\rm ac,\alpha}$ converges weakly to $\delta_0$ as $\alpha\nearrow1$.
This finishes the proof of Theorem B.1.
\end{proof}

\bibliographystyle{amsplain}

\end{document}